\documentclass[preprint,12pt]{elsarticle}
\usepackage{lineno,hyperref}
\usepackage{epsfig}
\usepackage{multirow}
\usepackage{epstopdf}
\usepackage{graphics}
\usepackage{graphicx}
\usepackage{makecell}
\usepackage{subfigure}
\usepackage{amssymb}
\usepackage{amsmath}
\usepackage{mathrsfs}
\usepackage{caption}
\usepackage{colortbl}
\usepackage{amsfonts}
\usepackage{bm}
\usepackage{caption}
\biboptions{sort&compress}
\newtheorem{remark}{Remark}
\newtheorem{theorem}{Theorem}
\newtheorem{lemma}{Lemma}
\newtheorem{definition}{Definition}
\newtheorem{proof}{Proof}
\newtheorem{assumption}{Assumption}

\begin{document}
	\begin{frontmatter}
		\title{Parareal Algorithms for Stochastic Maxwell Equations Driven by Multiplicative Noise}
		\author[1]{Liying Zhang}
		\ead{zhangliying@cumtb.edu.cn}
		\author[1]{Qi Zhang\corref{cor1}}
		\ead{zq17866703992@163.com}
		\author[b,c]{Lihai Ji} 
		\ead{jilihai@lsec.cc.ac.cn}
		\cortext[cor1]{Corresponding author.}
		\address[1]{School of Mathematical Science, China University of Mining and Technology, Beijing 100083, China}
		\address[b]{Institute of Applied Physics and Computational Mathematics, Beijing 100094, China}
		\address[c]{Shanghai Zhangjiang Institute of Mathematics, Shanghai 201203, China}

		\begin{abstract}
			This paper investigates the parareal algorithms for solving the stochastic Maxwell equations driven by multiplicative noise, focusing on their convergence, computational efficiency and numerical performance. The algorithms use the stochastic exponential integrator as the coarse propagator, while both the exact integrator and the stochastic exponential integrator are used as fine propagators. Theoretical analysis shows that the mean square convergence rates of the two algorithms selected above are proportional to \(k/2\), depending on the iteration number of the algorithms. Numerical experiments validate these theoretical findings, demonstrating that larger iteration numbers \(k\) improve convergence rates, while larger damping coefficients \(\sigma\) accelerate the convergence of the algorithms. Furthermore, the algorithms maintain high accuracy and computational efficiency, highlighting their significant advantages over traditional exponential methods in long-term simulations.
		\end{abstract}
		\begin{keyword}
			Stochastic Maxwell equations \sep
			Parareal algorithm \sep
			Strong convergence \sep Stochastic exponential integrator
		\end{keyword}
		
	\end{frontmatter}
	
	\section{Introduction}\label{sec1}
	In the evolution of complex dynamic systems, the flux of electric and magnetic fields is influenced by noise, resulting in uncertainty and random effects that significantly impact the system's behavior. To accurately model thermal fluctuations and radiation in electromagnetic fields, stochastic Maxwell equations were introduced by \citep{Rytovetal1989}. Several theoretical studies on stochastic Maxwell equations, such as the well-posedness, homogenization, and controllability of the solutions have been studied (cf. \citep{Liaskosetal2010,Horsinetal2010,Roachetal2012}). 
	
	In recent years, significant research has focused on numerical methods for discretizing stochastic Maxwell equations to effectively explore the physical properties of the solutions. Several studies have concentrated on developing structure-preserving numerical methods, such as stochastic multi-symplectic numerical methods (cf. \citep{Hongetal2014,Chenetal2016,Hongetal2017,Zhangetal2019a}),  discontinuous Galerkin methods (cf. \citep{Chen2021,Sunetal2022,Sunetal2023}), local and global radial basis functions (cf. \citep{Hongetal2022,Hou2023}),  symplectic Runge-Kutta methods  (cf. \citep{Chenetal2019a}), ergodic numerical
	method (cf. \citep{Chenetal2022}) and operator splitting method (cf. \citep{Chenetal2021}). Additionally, other discretization methods include implicit Euler scheme (cf. \citep{Chenetal2019b}), the explicit exponential scheme (cf. \citep{Cohenetal2020}), the finite element method (cf. \citep{Zhang2008}), CN-FDTD and Yee-FDTD methods   (cf. \citep{ZhouLiang2024}), Wiener chaos expansion (cf. \citep{Badieirostamietal2010}), among others. Given that the convergence rates of numerical methods for both temporal and spatial discretization can be influenced by the regularity of the noise, it is challenging to establish effective numerical methods. For instance, the semi-implicit Euler method achieves the mean-square convergence order of \(1/2\) for multiplicative noise in \citep{Chenetal2019b}, while the stochastic Runge-Kutta method attains the mean-square convergence order of \(1\) for additive noise in \citep{Chenetal2019a}. High-order DG methods exhibit the mean-square convergence order of \(k+1\) for both multiplicative and additive noise in \citep{Sunetal2022, Sunetal2023} and the explicit exponential integrator achieves the mean-square convergence order  of \(1/2\) and \(1\) for multiplicative and additive noise, respectively in \citep{Cohenetal2020}. 
	
   The aim of this paper is to design an efficient numerical method to perform long-time and high-precision numerical simulations of stochastic Maxwell equations driven by multiplicative noise.  In recent years, parareal algorithms have been applied to the numerical computation of stochastic differential equations. For example, Zhang et al. in \citep{ZhangWang2020} proposed the parareal algorithm based on the explicit Milstein scheme as the coarse propagator and the exact solution as the fine propagator, achieving the mean-square convergence order of $k$. Hong et al. in \citep{Hongetal2019}  employed an exponential $\theta$-scheme to solve stochastic Schr\"odinger equations driven by additive noise, achieving the mean-square convergence order of $k$ for $\theta \in [0,1] \backslash {\frac{1}{2}}$ and a higher order of $2k$ when $\theta = \frac{1}{2}$ owing to algorithmic symmetry. Br\'ehier et al. in \citep{BrehierWang2020} investigated the parareal algorithm for semilinear parabolic stochastic partial differential equations, proving that when the linear implicit Euler scheme is chosen as the coarse integrator, the mean-square convergence order is $\min(\alpha,k+1)$, saturating at  $\alpha$ as $k$ increases, whereas using the exponential Euler scheme as the coarse integrator yields the mean-square convergence order $(k+1)\alpha$. Consequently, the stochastic exponential scheme typically outperforms the implicit Euler method in terms of convergence rates. In \citep{ZhangZhang}, we studied parareal algorithms for stochastic Maxwell equations driven by additive noise, selecting the stochastic exponential integrator  as the coarse propagator and both the exact solution integrator and the stochastic exponential integrator as the fine propagators and derived that the uniform mean-square convergence order is $k$. 
    
   In this paper, we aim to investigate the parareal algorithms for stochastic Maxwell equations driven by multiplicative noise. When analyzing the convergence of parareal algorithms for stochastic Maxwell equations driven by multiplicative noise, compared to additive noise in \citep{ZhangZhang}, there are several challenges. Since additive noise affects the system independently of its state and does not depend on state variables, semigroup contraction properties and Lipschitz continuity can be more directly leveraged in the convergence analysis. In contrast, multiplicative noise dynamically interacts with the solution, meaning that the noise term varies with the state variables. Consequently, the infinite-dimensional Itô integral and properties of the $Q$-Wiener process need to be employed, introducing more intricate integral terms and estimation procedures. Moreover, establishing the error recursion requires the diffusion coefficient to satisfy the global Lipschitz condition for multiplicative noise, resulting in more complex recursive inequalities and iterative techniques, which are essential to guarantee the final convergence estimates. In numerical implementation, multiplicative noise may cause instability in numerical methods, requiring the selection of appropriate step sizes and parameters to ensure numerical stability. Additionally, the comparison of the computational cost and efficiency of the parareal algorithm and exponential methods reveals that traditional exponential methods often require smaller time steps to ensure accuracy, which in turn leads to higher computational costs. In contrast, the parareal algorithm uses a coarse propagator with a larger time step and a fine propagator with a smaller time step, allowing it to achieve comparable accuracy in fewer iterations, thus reducing overall computational effort. Finally, numerical experiments are presented to validate the convergence and efficiency of the parareal algorithm for the stochastic Maxwell equations driven by multiplicative noise.
    
	The structure of this paper is as follows. Section \ref{sec2} introduces the fundamentals of stochastic Maxwell equations, including Hilbert spaces, \(Q\)-Wiener processes and the Maxwell operator, and establishes the well-posedness of mild solutions. Section \ref{sec3} presents the parareal algorithms, which use the stochastic exponential integrator as the coarse \( \mathcal{G} \)-propagator and offer two options for the fine \( \mathcal{F} \)-propagator: the exact solution and the stochastic exponential integrator. Section \ref{sec4} analyzes the computational cost and efficiency, comparing the algorithm with traditional exponential methods and highlighting its ability to significantly reduce costs. Sections \ref{sec5} and \ref{sec6} present convergence analyses for two choices of the fine propagator, proving that the uniform mean-square convergence rate of the parareal algorithm is proportional to \( k/2 \), where \( k \) is the iteration number. Section \ref{sec7} validates the theoretical results through numerical experiments, demonstrating the stability, accuracy and efficiency of the algorithms in long-term simulations.
	
	To lighten notations, throughout this paper,  C stands for a constant which might be dependent of $T$ but is independent of $\Delta T$ and may vary from line to line.
	\section{Stochastic Maxwell equations driven by multiplicative noise} \label{sec2}
	
	In this paper, we consider the stochastic Maxwell equations with damping terms driven by multiplicative noise:
	
	\begin{align*}
		\varepsilon \partial_t \bm{E}(t,\bm{x}) &= \nabla \times \bm{H}(t,\bm{x}) - \sigma\bm{E}(t,\bm{x}) - J_e (t,\bm{x},\bm{E},\bm{H}) - J_{e}^r(t,\bm{x},\bm{E},\bm{H}) \cdot \dot{W}, \\
		\mu \partial_t \bm{H}(t,\bm{x}) &= -\nabla \times \bm{E}(t,\bm{x}) - \sigma\bm{H}(t,\bm{x}) - J_m (t,\bm{x},\bm{E},\bm{H}) - J_{m}^r(t,\bm{x},\bm{E},\bm{H}) \cdot \dot{W},
	\end{align*}
	where \( t \in [0,T] \), \( \bm{x} \in D \subset \mathbb{R}^{3} \), and \( \bm{E} \) and \( \bm{H} \) represent the electric and magnetic fields, respectively. Here, \( \varepsilon \) is the electric permittivity, \( \mu \) is the magnetic permeability, \( J_e \) and \( J_m \) are the electric and magnetic current densities, and \( J_{e}^r \) and \( J_{m}^r \) are functions depending on \( \bm{E} \) and \( \bm{H} \).
	
	The basic Hilbert space is defined as \( \mathbb{H}:= L^{2}(D)^{3} \times L^{2}(D)^{3} \) with the inner product:
	
	\begin{align*}
		\left\langle
		\begin{pmatrix}
			\bm{E}_1 \\
			\bm{H}_1
		\end{pmatrix},
		\begin{pmatrix}
			\bm{E}_2 \\
			\bm{H}_2
		\end{pmatrix}
		\right\rangle_\mathbb{H}
		&= \int_D (\varepsilon \bm{E}_1 \cdot \bm{E}_2 + \mu \bm{H}_1 \cdot \bm{H}_2) d\bm{x},
	\end{align*}
	for all \( \bm{E}_1, \bm{H}_1, \bm{E}_2, \bm{H}_2 \in L^{2}(D)^{3} \), and the norm:
	
	\begin{align*}
		\left\|
		\begin{pmatrix}
			\bm{E} \\
			\bm{H}
		\end{pmatrix}
		\right\|_\mathbb{H}^2
		&= \int_D (\varepsilon \|\bm{E}\|^2 + \mu \|\bm{H}\|^2) d\bm{x}, \quad \forall \bm{E}, \bm{H} \in L^{2}(D)^{3}.
	\end{align*}
	
	In addition, we assume that \( \varepsilon \) and \( \mu \) are bounded and uniformly positive definite functions: \( \varepsilon, \mu \in L^{\infty}(D) \), with \( \varepsilon, \mu > 0 \) for all \( \bm{x} \in D \).
	
	The \( Q \)-Wiener process \( W \) is defined on a given probability space \( (\Omega, \mathscr{F}, P, \{\mathscr{F}_t\}_{t \in [0,T]}) \) and can be expanded in a Fourier series as:
	
	\begin{align*}
		W(t) = \sum_{n=1}^{\infty} \lambda_n^{1/2} \beta_n(t) e_n, \quad t \in [0,T],
	\end{align*}
	where \( \{\beta_n(t)\}_{n=1}^{\infty} \) is a sequence of independent standard real-valued Wiener processes, and \( \{e_n\}_{n=1}^{\infty} \) is a complete orthonormal system of \( \mathbb{H} \), consisting of eigenfunctions of a symmetric, nonnegative, and finite-trace operator \( Q \), i.e., \( \text{Tr}(Q) < \infty \) and \( Q e_n = \lambda_n e_n \) with corresponding eigenvalues \( \lambda_n \geq 0 \).
	
	The Maxwell operator is defined by:
	
	\begin{equation}\label{eq1}
		M \begin{pmatrix}
			\bm{E} \\
			\bm{H}
		\end{pmatrix}
		:= \begin{pmatrix}
			0 & \varepsilon^{-1} \nabla \times \\
			-\mu^{-1} \nabla \times & 0
		\end{pmatrix}
		\begin{pmatrix}
			\bm{E} \\
			\bm{H}
		\end{pmatrix},
	\end{equation}
	with domain:
	\begin{equation*}
		\mathcal{D}(M) = \left\{
		\begin{pmatrix}
			\bm{E} \\
			\bm{H}
		\end{pmatrix}
		\in \mathbb{H} : M \begin{pmatrix}
			\bm{E} \\
			\bm{H}
		\end{pmatrix}
		= \begin{pmatrix}
			\varepsilon^{-1} \nabla \times \bm{H} \\
			-\mu^{-1} \nabla \times \bm{E}
		\end{pmatrix}
		\in \mathbb{H}, \, \bm{n} \times \bm{E} \bigg|_{\partial D} = 0 \right\}.
	\end{equation*}
	
	Based on the closedness of the operator \( \nabla \times \), we have the following lemma.
	
	\begin{lemma}\label{lemma1} \citep{Chenetal2023}
		The Maxwell operator defined in (\ref{eq1}) with domain \( \mathcal{D}(M) \) is closed and skew-adjoint, and generates a \( C_0 \)-semigroup \( S(t) = e^{tM} \) on \( \mathbb{H} \) for \( t \in [0,T] \). Moreover, the frequently used property for the Maxwell operator \( M \) is: \( \langle M u, u \rangle_{\mathbb{H}} = 0 \).
	\end{lemma}
	
	We consider the abstract form in the infinite-dimensional space \( \mathbb{H} := L^2(D)^3 \times L^2(D)^3 \):
	
	\begin{align}
		\label{eq2}
		\left\{
		\begin{array}{ll}
			d u(t) = [M u(t) - \sigma u(t)] dt + F(t,u(t)) dt + B(t,u(t)) dW, \quad t \in (0,T], \\
			u(0) = u_0,
		\end{array}
		\right.
	\end{align}
	where the solution \( u = (\bm{E}^T, \bm{H}^T)^T \) is a stochastic process with values in \( \mathbb{H} \).
	
	Let $\widehat S(t):=e^{ t(M-\sigma Id)}$ be the semigroup generated by operator $M-\sigma Id$. The following lemma states that the semigroup $\widehat S(t):=e^{ t(M-\sigma Id)}$ is a contraction semigroup, implying its properties of contraction, stability, and extensive applicability.
	\begin{lemma}\label{lemma2}
		For the semigroup $\{\widehat S(t)=e^{ t(M-\sigma Id)}, \, t\geq 0\}$ on $\mathbb{H}$, we obtain
		\begin{align*}
			\left\|\widehat S(t)\right\|_{\mathcal{L(\mathbb{H})}}\leq 1, \quad t \geq 0.
		\end{align*}	
	\end{lemma}
	To ensure the well-posedness of the mild solution of the stochastic Maxwell equations (\ref{eq2}), we need the following assumptions.
	
	\begin{assumption}\label{assump1}{\bf(Initial value)}.
		The initial value \( u_0 \) satisfies:
		\begin{align*}
			\|u_0\|^2_{L_2(\Omega, \mathbb{H})} < \infty.
		\end{align*}
	\end{assumption}
	
	\begin{assumption}\label{assump2}{\bf(Drift nonlinearity)}.
		The drift operator \( F \) satisfies:
		\begin{align*}
			&\|F(t,u)\|_{\mathbb{H}} \leq C(1 + \|u\|_{\mathbb{H}}), \\
			&\|F(t,u) - F(s,v)\|_{\mathbb{H}} \leq C(|t - s| + \|u - v\|_{\mathbb{H}}),
		\end{align*}
		for all \( t,s \in [0,T], \, u, v \in \mathbb{H} \). Moreover, the nonlinear operator \( F \) has bounded derivatives:
		\begin{align*}
			\|DF(u).h\|_{\mathbb{H}} \leq C \|h\|_{\mathbb{H}},
		\end{align*}
		for \( h \in \mathbb{H} \), where the linear operator  \(DF(u)\)  stands for the Fréchet derivative of  \(F\)  at  \(u\)  and  \(DF(u).h\)  stands for the derivative of \(F\)  at  \(u\)  in the direction \(h\).
	\end{assumption}
	
	\begin{assumption}\label{assump3}{\bf(Diffusion nonlinearity)}.
		The diffusion operator \( B \) satisfies:
		\begin{align*}
			&\|B(t,u)\|_{HS(U_0, \mathbb{H})} \leq C \|Q^{1/2}\|_{HS}(1 + \|u\|_{\mathbb{H}}^2)^{1/2}, \\
			&\|B(t,u) - B(s,v)\|_{HS(U_0, \mathbb{H})} \leq C \|Q^{1/2}\|_{HS}(|t - s| + \|u - v\|_{\mathbb{H}}),
		\end{align*}
		for all \( t, s \in [0,T], \, u, v \in \mathbb{H}, U_0 := Q^{1/2} \mathbb{H} \). Moreover, the nonlinear operator \( B \) has bounded derivatives:
		\begin{align*}
			\|DB(u).h\|_{HS(U_0, \mathbb{H})} \leq C \|Q^{1/2}\|_{HS} \|h\|_{\mathbb{H}},
		\end{align*}
		for \( h \in \mathbb{H} \),  where the linear operator  \(DB(u)\)  stands for the Fréchet derivative of  \(B\)  at  \(u\)  and  \(DB(u).h\)  stands for the derivative of \(B\) at \(u\)  in the direction \(h\). 
	\end{assumption}
	
	\begin{assumption}\label{assump4}\cite{Yan2005}{\bf(Covariance operator)}. 
		To guarantee the existence of a mild solution, we further assume the covariance operator \( Q \) of \( W(t) \) satisfies:
		\begin{align*}
			\|M^{(\beta - 1)/2} Q^{1/2}\|_{\mathcal{L}_2(\mathbb{H})} < \infty, \quad \beta \in [0, 1],
		\end{align*}
		where \( \|\cdot\|_{\mathcal{L}_2(\mathbb{H})} \) denotes the Hilbert–Schmidt norm for operators from \( \mathbb{H} \) to \( \mathbb{H} \), and \( M^{(\beta - 1)/2} \) is the \( (\beta - 1)/2 \)-th fractional power of \( M \), with \( \beta \) being a parameter characterizing the regularity of the noise. In this article, we are mostly interested in \( \beta = 1 \) for the trace-class operator \( Q \).
	\end{assumption}
	
	\begin{lemma}\label{lemma2}\citep{Chenetal2023}
		Let Assumptions \ref{assump1}, \ref{assump2}, \ref{assump3}, and \ref{assump4} hold. Then there exists a unique mild solution to (\ref{eq2}), which satisfies:
		\begin{equation*}
			u(t) = \widehat{S}(t) u_0 + \int_0^t \widehat{S}(t - s) F(s, u(s)) ds + \int_0^t \widehat{S}(t - s) B(s, u(s)) dW(s), \quad \mathbb{P}\text{-a.s.},
		\end{equation*}
		for each \( t \in [0,T] \), where \( \widehat{S}(t) = e^{t(M - \sigma I)} \) is a \( C_0 \)-semigroup generated by \( M - \sigma I \).
	\end{lemma}
	\section{Temporal semidiscretization by parareal algorithm} \label{sec3}
	\subsection{Framework of parareal algorithm}
	To perform the parareal algorithm, the considered interval $[0, T]$ is first divided into $N$ time intervals $[t_{n-1},t_{n}]$  with a uniform coarse step-size $\Delta T=t_{n}-t_{n-1}$ for any $n=1,\cdots,N$. Each subinterval is further divided into $J$ smaller time intervals $[t_{n-1,j-1}, t_{n-1,j}]$ with a uniform fine step-size $\Delta t = t_{n-1,j} - t_{n-1,j-1}$, where $n = 1, \dots, N$ and $j = 1, \dots, J$. The parareal algorithm can be described as follows.
	\begin{itemize}
		\item Initialization.	Use the coarse propagator $\mathcal{G}$ with the coarse step-size $\Delta T$ to compute initial value $u_{n}^{(0)}$ by
		\begin{align*}
			u_{n}^{(0)}&=\mathcal{G} (t_{n-1},t_{n},u_{n-1}^{(0)}),\quad n =1,\ldots ,N.\\
			u_{0}^{(0)}&=u_{0}.
		\end{align*}
		
		Let $K \in \mathbb{N}$ denote the number of parareal iterations: for all $n= 0,\ldots,N, \,k =0,\ldots,K-1$.
		\item Time-parallel computation.  Use the fine propagator $\mathcal{F}$ and  time step-size  $\Delta t$ to compute $\widehat u_{n}$ on each subinterval $[t_{n-1},t_{n}]$  independently
		\begin{align}\label{eq3}
			&\widehat u_{n-1,j}=\mathcal{F}(t_{n-1,j-1},t_{n-1,j},\widehat u_{n-1,j-1}), \quad j=1,\cdots,J ,\\
			&\widehat u_{n-1,0}=u_{n-1}^{(k)}.\nonumber
		\end{align}
		
		\item Prediction and correction. Note that we obtain two numerical solutions $u_{n}^{(0)}$ and $\widehat u_{n-1,J}$
		at time $t_{n}$ through initialization and parallelization, the sequential
		prediction and correction is defined as
		\begin{align}\label{eq4}
			u_{n}^{(k+1)}&=\mathcal{G} (t_{n-1},t_{n},u_{n-1}^{(k+1)})+\widehat u_{n}-\mathcal{G} (t_{n-1},t_{n},u_{n-1}^{(k)}),\\
			u_{0}^{(k)}&=u_{0}.\nonumber
		\end{align}
		Noting  that equation (\ref{eq3}) is of the following form $\widehat u_{n}=\mathcal{F}(t_{n-1},t_{n},u_{n-1}^{(k)})$, then parareal algorithm can be written as 
		\begin{align}\label{eq5}
			u_{n}^{(k+1)}=\mathcal{G} (t_{n-1},t_{n},u_{n-1}^{(k+1)})+\mathcal{F}(t_{n-1},t_{n},u_{n-1}^{(k)})-\mathcal{G} (t_{n-1},t_{n},u_{n-1}^{(k)}).
		\end{align}
	\end{itemize}
	\subsection{Stochastic exponential scheme} 
	Consider the mild solution of the stochastic Maxwell equations (\ref{eq2}) on the time interval $[t_{n-1},t_{n}]$ 
	\begin{align}\label{eq6}
		u(t_{n})=\widehat S(\Delta T)u(t_{n-1})+\int_{t_{n-1}}^{t_{n}}\widehat S(t_{n}-s)F(u(s))ds+\int_{t_{n-1}}^{t_{n}}\widehat S(t_{n}-s)B(u(s))dW(s),
	\end{align} 
	where $C_{0}$-semigroup $\widehat S(\Delta T)=e^{\Delta T(M-\sigma Id) }$.
	
	By approximating the integrals of the mild solution (\ref{eq6}) at the left endpoints, we can obtain the stochastic exponential scheme
	\begin{align}\label{eq7}
		u_{n}=\widehat S(\Delta T)u(t_{n-1})+\widehat S(\Delta T)F(u(t_{n-1}))\Delta T+\widehat S(\Delta T)B(u(t_{n-1}))\Delta W_{n},
	\end{align}
	where $\Delta W_{n}=W(t_{n})-W(t_{n-1})$.
	\subsection{Coarse and fine propagators}
	\begin{itemize}
		\item Coarse propagator.
		The stochastic exponential scheme is chosen as
		the coarse propagator  with
		time step-size $\Delta T$ by (\ref{eq7})
		\begin{align}\label{eq8}
			\mathcal{G}(t_{n-1},t_{n},u)=\widehat S(\Delta T)u+\widehat S(\Delta T)F(u)\Delta T+\widehat S(\Delta T)B(u)\Delta W_{n}.
		\end{align}
		\item Fine propagator.
		The exact solution as the fine propagator with time step-size $\Delta t$ by (\ref{eq6}) 
		\begin{align}\label{eq9}
			\!\!\!\!\!\!\!\!\!\mathcal{F}(t_{n-1,j-1},t_{n-1,j},u)=\widehat S(\Delta t)u\!\!+\!\!\int_{0}^{\Delta t}\!\!\widehat S(\Delta t-s)F(u(s))ds
			\!\!+\!\!\int_{0}^{\Delta t}\!\!\!\widehat S(\Delta t-s)B(u(s))dW(s).
		\end{align}
		
		Besides, the other choice is the stochastic exponential scheme is chosen as the fine propagator with time step-size $\Delta t$ by (\ref{eq7})
		\begin{align}\label{eq10}
			\mathcal{F}(t_{n-1,j-1},t_{n-1,j},u)=\widehat S(\Delta t)u+\widehat S(\Delta t)F(u)\Delta t+\widehat S(\Delta t)B(u)\Delta W_{n-1,j},
		\end{align}
		where  $\Delta W_{n-1,j}=W(t_{n-1,j})-W(t_{n-1,j-1})$.
	\end{itemize}
	\section{Computational cost and efficiency analysis}\label{sec4}
	In this section, we provide a detailed analysis of the computational cost associated with the parareal algorithm and the exponential solution. Additionally, we derive the efficiency of the parareal algorithm compared to the exponential solution.
	
	\subsection{Computational cost of the parareal algorithm}
	The advantage of the parareal algorithm lies in the ability to compute $u_n^{(k+1)}$ in parallel (for $k \geq 0$), reducing the computational cost. We define the following variables:
	\begin{itemize}
		\item  $N_{\text{proc}}$ denotes the number of available processors;
		\item $T$ denotes the total time, such that $N \Delta T = T$;
		\item $\tau_G$ denotes the computational time for a single evaluation of $G(u, t_n, t_{n+1})$;
		\item $\tau_{F,\text{aux}}$ denotes the computational time of the auxiliary fine integrator $F_{\text{aux}}(u, t_{n,j}, t_{n,j+1})$;
		\item $\tau_F$ denotes the computational time for a single evaluation of $F(\cdot, t_n, t_{n+1})$, such that $\tau_F = J \tau_{F,\text{aux}} = \frac{\Delta T}{\delta T} \tau_{F,\text{aux}}$.
	\end{itemize}
	
	It is assumed that $\tau_G$ and $\tau_F$ do not depend on $\Delta T$, $\delta t$, $n$, or $u$.
	
	The computational cost of the parareal algorithm consists of two main components: the initialization step and the iterative process. 
	\begin{enumerate}
		\item Initialization Cost: In the initialization step, the coarse propagator $\mathcal{G}$ is applied sequentially over $N$ sub-intervals, resulting in the following computational cost:
		\[
		\text{Cost}_{\text{init}}= N \tau_G \quad \text{(only for \( k = 0 \))}.
		\]
		\item Iteration Cost: In each iteration, the coarse propagator $\mathcal{G}$ is sequentially evaluated, and the fine propagator $\mathcal{F}$ is evaluated in parallel across  $N_{\text{proc}}$ processors. The computational cost of one iteration is:
		\[
		\text{Cost}_{\text{iter}}=  N \left( \tau_G + \frac{\tau_F}{N_{\text{proc}}} \right).
		\]
		The total cost for \( K \) iterations is:
		\[
		K \times N \left( \tau_G + \frac{\tau_F}{N_{\text{proc}}} \right).
		\]
		\item Total Cost: For $K$ iterations of the parareal algorithm, the total computational cost is:
		\[
		\text{Cost}^{\text{parareal}} = (K+1) \frac{T}{\Delta T} \tau_G + K \frac{T}{\Delta T}\frac{\Delta T}{\delta t} \frac{\tau_ {F,\text{aux}}}{N_{\text{proc}}},
		\]
	\end{enumerate}
	where
	\begin{itemize}
		\item The first term $(K+1) \frac{T}{\Delta T} \tau_G$
		\begin{itemize}
			\item Represents the computational cost of the coarse integrator during initialization and all $K$ iterations.
			\item $\frac{T}{\Delta T}$ is the total number of coarse time steps.
		\end{itemize}
		
		\item The second term $K \frac{T}{\Delta T}\frac{\Delta T}{\delta t} \frac{\tau_ {F,\text{aux}}}{N_{\text{proc}}}= K \frac{T}{\delta t} \frac{\tau_{F,\text{aux}}}{N_{\text{proc}}}$
		\begin{itemize}
			\item Represents the computational cost of the fine propagator $\mathcal{F}$ over all iterations;
			\item $\frac{T}{\delta t}$ is the total number of fine time steps.
		\end{itemize}
	\end{itemize}
	
	\begin{remark}
		In the $k$-th iteration, $G(u^{(k)}, t_n, t_{n+1})$ has already been computed in the previous iteration. Therefore, only the coarse propagator $G(u^{(k+1)}, t_n, t_{n+1})$ needs to be reevaluated. The computation of $F(u^{(k)}, t_n, t_{n+1})$ can be performed in parallel, as the calculations for different time steps $n$ are independent. However, the total cost is influenced by the number of iterations $K$, the number of processors  $N_{\text{proc}}$, and the choice of time step sizes $\Delta T$ and $\delta t$.
	\end{remark}
	
	\subsection{Computational cost of the exponential solution}
	Let $\tau^{\text{exp}}$ denote the computational time for the exponential solution $ u_n^{\text{exp}}$ with time step $\Delta T^{'}$. The computational cost is:
	\[
	\text{Cost}^{\text{exp}} = \frac{T}{\Delta T^{'}} \tau^{\text{exp}},
	\]
	where $\tau^{\text{exp}}$ represents the computational cost of the exponential solution per sub-interval in the exponential solution.
	
	\subsection{Computational efficiency}
	The computational efficiency $\mathcal{E}$ is defined as the ratio of the cost of the exponential solution to the cost of the parareal algorithm:
	\[
	\mathcal{E} = \frac{	\text{Cost}^{\text{exp}}}{\text{Cost}^{\text{parareal}}} = \frac{\tau^{\text{exp}} / \Delta T'}{(K + 1) \tau_G / \Delta T + K \tau_F /N_{\text{proc}} \delta t}.
	\]
	
	\begin{remark}
		The parareal algorithm demonstrates outstanding performance in reducing computational costs and enhancing efficiency. However, a balance must be struck between efficiency and the number of iterations $K$. Selecting an appropriate $K$ and time step size is essential for achieving optimal performance. For the maximum efficiency, it is recommended to set $K = 1$.
	\end{remark}
	
	\subsection{Efficiency analysis}
	The efficiency is influenced by the following factors.
	
	\begin{enumerate}
		\item Time Step Ratios:  
		The ratios  $\Delta T / \delta t$  and $\Delta T / \Delta T'$ are critical for $\mathcal{E}$.
		\begin{itemize}
			\item $\Delta T / \delta t$: The ratio $\Delta T / \delta t$ significantly influences the efficiency of the parareal algorithm. A smaller coarse time step $ \Delta T$ improves the accuracy of the coarse propagator $\mathcal{G}$, potentially reducing the number of iterations $K$ and accelerating convergence. However, this can increase initialization costs due to the need for more coarse evaluations. Conversely, a smaller fine time step $\delta t$ enhances the precision of the fine propagator $\mathcal{F}$, but also raises the computational costs for each evaluation. If the ratio of coarse to fine time steps is too large, it can lead to disproportionately high costs for the fine propagator, negatively affecting overall efficiency.
			\item $\Delta T / \Delta T'$: This ratio reflects the comparison between the parareal algorithm and traditional exponential methods. To achieve the same level of accuracy, the time step $\Delta T'$ required by exponential methods is typically much smaller than the coarse time step $\Delta T$ used in the parareal algorithm. This is because traditional exponential methods often necessitate smaller time steps $\Delta T'$ to more accurately approximate the solution, which consequently leads to a significant increase in computational costs. In contrast, the parareal algorithm decomposes the computational task into two levels: it allows for an initial estimate using a larger coarse time step $\Delta T$, followed by refinement through the use of a finer time step $\delta t$.
		\end{itemize}
		
		\item Number of Processors ($N_{\text{proc}}$):  
		Increasing $N_{\text{proc}}$ directly reduces the computational time of the fine propagator $\mathcal{F}$ by distributing the workload, thereby enhancing $\mathcal{E}$.
		\item Number of Iterations ($ K $): The number of iterations $ K $ affects both computational cost and solution accuracy. Each additional iteration linearly increases the computational cost. Fewer iterations can reduce costs but may compromise accuracy. An optimal $ K $ should strike a balance between efficiency and solution precision.
	\end{enumerate}
	
	In the following two sections, two convergence analysis results will be given, i.e., we investigate the parareal algorithms obtained by choosing the stochastic exponential integrator as the coarse integrator and both the exact integrator and the stochastic exponential integrator as the fine integrator.
	\section{Error estimate of exact integrator as the fine integrator }\label{sec5}
	\begin{theorem}\label{theorem1}
		Let Assumption \ref{assump1}, \ref{assump2}, \ref{assump3} and \ref{assump4} hold,  we apply the stochastic exponential integrator for coarse propagator $\mathcal{G}$ and the exact solution integrator for fine propagator $\mathcal{F}$. Then we have the following convergence estimate for the fixed iteration number $k$
		\begin{equation}\label{eq11}
			\sup\limits_{1\leq n \leq N}\left\|u(t_{n})-u_{n}^{(k)}\right\|_{L_{2}(\Omega,\,\mathbb{H})}\leq C_{T,N} \Delta T^{k/2} \sup\limits_{1\leq n \leq N}\left\|u(t_{n})-u_{n}^{(0)}\right\|_{L_{2}(\Omega,\, \mathbb{H})},
		\end{equation}
		with a positive constant $C$ independent on $\Delta T$, where the parareal solution $u_{n}^{(k)}$ is defined in (\ref{eq5}) and the exact solution $u(t_{n})$ is defined in (\ref{eq6}).
	\end{theorem}
	
	To simplify the exposition, let us introduce the following notation.
	\begin{definition}
		The residual operator 
		\begin{align}\label{eq12}
			\mathcal{R}(t_{n-1},t_{n},u):=\mathcal{F}(t_{n-1},t_{n},u)-\mathcal{G}(t_{n-1},t_{n},u),
		\end{align}
		for all $n\in{0,\cdots,N}$.
	\end{definition}
	
	Before the error analysis, the following two important lemmas are introduced.
	\begin{lemma}\label{lemma3}
		Let $M:=M(\beta)_{N\times N}$ be a strict lower triangular Toeplitz matrix and its elements are defined as 
		\begin{align*}
			M_{i1}=
			\left\{
			\begin{array}{ll}
				0, & \quad i=1,\\
				\beta^{i-2}, & \quad 2\leq i\leq N.
			\end{array}
			\right.
		\end{align*}
		The infinity norm of the $kth$ power of $M$ is bounded as follows
		\begin{align*}
			\left\|M^{k}(\beta)\right\|_{\infty}\leq
			\left\{
			\begin{array}{ll}
				min
				\begin{Bmatrix}
					\left(\dfrac{1-|\beta|^{N-1}}{1-|\beta|}\right)^{k},
					\begin{pmatrix}
						N-1\\
						k
					\end{pmatrix}
				\end{Bmatrix}, &\quad |\beta|<1,\\
				|\beta|^{N-k-1}
				\begin{pmatrix}
					N-1\\
					k
				\end{pmatrix} ,&\quad |\beta| \geq 1.
			\end{array}
			\right.
		\end{align*}
	\end{lemma}
	\begin{lemma}\label{lemma4}
		Let $\gamma,\eta\geq0$, a double indexed sequence \{$\delta_{n}^{k}$\} satisties $\delta_{n}^{k}\geq 0,\,\delta_{0}^{k}\geq 0$ and
		\begin{equation*}
			\delta_{n}^{k}\leq \gamma\delta_{n-1}^{k}+\eta\delta_{n-1}^{k-1},
		\end{equation*}
		for $n=0,1,\cdots,N$ and $k=0,1,\cdots,K$, then vector $\zeta^{k}=(\delta_{1}^{k},\delta_{2}^{k},\cdots,\delta_{N}^{k})^{T} $satisfies
		\begin{align*}
			\zeta^{k}\leq \eta M(\gamma)\zeta^{k-1}.
		\end{align*}
	\end{lemma}
	\begin{proof}
		Since the exact solution  $u(t_{n})$  is chosen as the fine propagator $\mathcal{F}$, it can be written as
		\begin{align}\label{eq13}
			u(t_{n})&=\mathcal{F}(t_{n-1},t_{n},u(t_{n-1}))\nonumber\\
			&=\mathcal{G} (t_{n-1},t_{n},u(t_{n-1}))+\mathcal{F} (t_{n-1},t_{n},u(t_{n-1}))-\mathcal{G} (t_{n-1},t_{n},u(t_{n-1})).
		\end{align}
		Subtracting (\ref{eq5}) from (\ref{eq13}) and  using the notation of the residual operator (\ref{eq12}), we obtain
		\begin{align*}
			E\|u(t_{n})-u_{n}^{(k)}\|_{\mathbb{H}}^2&\leq C \{E\left\|\mathcal{G} (u(t_{n-1}))-\mathcal{G} (u_{n-1}^{(k)})\right\|_{\mathbb{H}}^2+E\left\| \mathcal{R}(u(t_{n-1}))
			- \mathcal{R} (u_{n-1}^{(k-1)})\right\|_{\mathbb{H}}^2\}\\
			&:=C\{I_{1}+I_{2}\}.
		\end{align*}
		Firstly, we estimate $I_{1}$. Applying the stochastic exponential integrator (\ref{eq7}) for the coarse propagator $\mathcal{G}$, it holds that
		\begin {align}
		\mathcal{G} (u(t_{n-1}))&=\widehat S(\Delta T)u(t_{n-1})+\widehat S(\Delta T)F(u(t_{n-1}))\Delta T+\widehat S(\Delta T) B(u(t_{n-1}))\Delta W_{n},\label{eq14}\\
		\mathcal{G} (u_{n-1}^{(k)})&=\widehat S(\Delta T)u_{n-1}^{(k)}+\widehat S(\Delta T)F(u_{n-1}^{(k)})\Delta T+\widehat S(\Delta T)B(u_{n-1}^{(k)})\Delta W_{n}.\label{eq15}
	\end{align}
	Subtracting the above two formulas  leads to
	\begin{align}\label{eq16}
		I_{1}&\leq C\{E\left\|\widehat S(\Delta T)(u(t_{n-1})-u_{n-1}^{(k)})\right\|_{\mathbb{H}}^2+E\left\|\widehat S(\Delta T)(F(u(t_{n-1}))-F(u_{n-1}^{(k)}))\Delta T\right\|_{\mathbb{H}}^2\nonumber\\
		+&E\left\|\widehat S(\Delta T)(B(u(t_{n-1}))-B(u_{n-1}^{(k)}))\Delta W_{n}\right\|_{\mathbb{H}}^2\} \nonumber\\
		&\leq  C\{\left\|\widehat S(\Delta T)\right\|_{\mathbb{H}}^2E\left\|u(t_{n-1})-u_{n-1}^{(k)}\right\|_{\mathbb{H}}^2+\Delta T^2\left\|\widehat S(\Delta T)\right\|_{\mathbb{H}}^2E\left\|u(t_{n-1})-u_{n-1}^{(k)}\right\|_{\mathbb{H}}^2\nonumber \\
		&+\Delta T\left\|Q^{\frac{1}{2}}\right\|_{HS}^2\left\|Q^{\frac{1}{2}}\right\|_{HS}^2\left\|\widehat S(\Delta T)\right\|_{\mathbb{H}}^2E\left\|u(t_{n-1})-u_{n-1}^{(k)}\right\|_{\mathbb{H}}^2\}\nonumber \\
		&\leq C(1+\Delta T^2+\Delta T)E\left\|u(t_{n-1})-u_{n-1}^{(k)}\right\|_{\mathbb{H}}^2\nonumber \\
		&\leq C(1+\Delta T)E\left\|u(t_{n-1})-u_{n-1}^{(k)}\right\|_{\mathbb{H}}^2
	\end{align}
	which by the contraction property of semigroup and the global Lipschitz property of $F$ and $B$.
	
	Now it remains to estimate $I_{2}$. Applying exact solution integrator (\ref{eq6}) for fine progagator $\mathcal{F}$ leads to
	\begin{align}
		\mathcal{F}(t_{n-1},t_{n},u(t_{n-1}))=\widehat S(\Delta T)u(t_{n-1})+\int_{0}^{\Delta T}\widehat S(\Delta T-s)F(U(t_{n-1},t_{n-1}+s,u(t_{n-1})))ds\nonumber\\
		+\int_{0}^{\Delta T}\widehat S(\Delta T-s)B(U(t_{n-1},t_{n-1}+s,u(t_{n-1})))dW(s),\label{eq17}\\
		\mathcal{F}(t_{n-1},t_{n},u_{n-1}^{(k-1)})=\widehat S(\Delta T)u_{n-1}^{(k-1)}+\int_{0}^{\Delta T}\widehat S(\Delta T-s)F(V(t_{n-1},t_{n-1}+s,u_{n-1}^{(k-1)}))ds\nonumber\\
		+\int_{0}^{\Delta T}\widehat S(\Delta T-s)B(V(t_{n-1},t_{n-1}+s,u_{n-1}^{(k-1)}))dW(s),\label{eq18}
	\end{align}
	where $U(t_{n-1},t_{n-1}+s,u)$ and $V(t_{n-1},t_{n-1}+s,u)$ denote the exact solution of system (\ref{eq2}) at time $t_{n-1}+s$ with the initial value $u$ and the initial time $t_{n-1}$.
	
	Substituting the above equations and equations (\ref{eq14}) and (\ref{eq15})  into the residual operator (\ref{eq12}), we obtain
	\begin{align*}
		I_{2}&=E\left\| \mathcal{R}(t_{n-1},t_{n},u(t_{n-1}))
		- \mathcal{R} (t_{n-1},t_{n},u_{n-1}^{(k-1)})\right\|_{\mathbb{H}}^2\\
		&\leq C\{ E\left\|\int_{0}^{\Delta T}\widehat S(\Delta T-s)[F(U(t_{n-1},t_{n-1}+s,u(t_{n-1})))-F(V(t_{n-1},t_{n-1}+s,u_{n-1}^{(k-1)}))]ds\right\|_{\mathbb{H}}^2\\
		&+E\left\|\int_{0}^{\Delta T}\widehat S(\Delta T-s)[B(U(t_{n-1},t_{n-1}+s,u(t_{n-1})))-B(V(t_{n-1},t_{n-1}+s,u_{n-1}^{(k-1)}))]dW(s)\right\|_{\mathbb{H}}^2\\
		&+E\left\|\widehat S(\Delta T)[F(u(t_{n-1}))-F(u_{n-1}^{(k-1)})]\Delta T\right\|_{\mathbb{H}}^2+E\left\|\widehat S(\Delta T)[B(u(t_{n-1}))-B(u_{n-1}^{(k-1)})]\Delta W_{n}\right\|_{\mathbb{H}}^2\}\\
		&:=C\{I_{3}+I_{4}+I_{5}+I_{6}\}.
	\end{align*}
	To get the estimation of $I_{3}$ and $I_{4}$, by Lipschitz continuity property for $F$ and $B$ , we derive
	\begin{align}\label{eq19}
		\!\!I_{3}&\leq \Delta T E\!\int_{0}^{\Delta T}\!\left\|\widehat S(\Delta T-s)\right\|_{\mathbb{H}}^2\left\|F(U(t_{n-1},t_{n-1}+s,u(t_{n-1})))-F(V(t_{n-1},t_{n-1}+s,u_{n-1}^{(k-1)}))\right\|_{\mathbb{H}}^2\!ds\nonumber\\
		&\leq C\Delta TE\int_{0}^{\Delta T} \left\|U(t_{n-1},t_{n-1}+s,u(t_{n-1}))-V(t_{n-1},t_{n-1}+s,u_{n-1}^{(k-1)})\right\|_{\mathbb{H}}^2ds\nonumber\\
		&\leq C\Delta T^2 E\left\|u(t_{n-1})-u_{n-1}^{(k-1)}\right\|_{\mathbb{H}}^2.
	\end{align}
	\begin{align}\label{eq20}
		\!\!\!I_{4}&\leq\!\left\|Q^{\frac{1}{2}}\right\|_{HS}^2\!\!\!E\!\!\int_{0}^{\Delta T} \!\!\!\left\|\widehat S(\Delta T-s)\right\|_{\mathbb{H}}^2\!\left\|B(U(t_{n-1},t_{n-1}+s,u(t_{n-1})))\!\!-\!\!B(V(t_{n-1},t_{n-1}+s,u_{n-1}^{(k-1)}))\right\|_{HS}^2\!\!\!\!\!\!\!ds\nonumber\\
		&\leq C \left\|Q^{\frac{1}{2}}\right\|_{HS}^2 \left\|Q^{\frac{1}{2}}\right\|_{HS}^2 E\int_{0}^{\Delta T} \left\|U(t_{n-1},t_{n-1}+s,u(t_{n-1}))-V(t_{n-1},t_{n-1}+s,u_{n-1}^{(k-1)})\right\|_{\mathbb{H}}^2ds\nonumber\\
		&\leq C \left\|Q\right\|_{HS}^2\Delta T E\left\|u(t_{n-1})-u_{n-1}^{(k-1)}\right\|_{\mathbb{H}}^2.
	\end{align}
	As for $I_{5}$ and $I_{6}$, using the contraction property of semigroup and Lipschitz continuity property for $F$ and $B$ yield
	\begin{align}\label{eq21}
		I_{5}&\leq C\Delta T^2 E\left\|u(t_{n-1})-u_{n-1}^{(k-1)}\right\|_{\mathbb{H}}^2.
	\end{align}
	
	\begin{align}\label{eq22}
		I_{6}&\leq C\Delta T\left\|Q^{\frac{1}{2}}\right\|_{HS}^2 E\left\|u(t_{n-1})-u_{n-1}^{(k-1)}\right\|_{\mathbb{H}}^2.
	\end{align}
	From (\ref{eq19}) (\ref{eq20}) (\ref{eq21}) and (\ref{eq22}), we know that
	\begin{align}\label{eq23}
		I_{2}&\leq C(\Delta T^2+\left\|Q^{\frac{1}{2}}\right\|_{HS}^2\Delta  T+\Delta T^2+\left\|Q^{\frac{1}{2}}\right\|_{HS}^2\Delta T)E\left\|u(t_{n-1})-u_{n-1}^{(k-1)}\right\|_{\mathbb{H}}^2\nonumber\\
		&\leq C\Delta TE\left\|u(t_{n-1})-u_{n-1}^{(k-1)}\right\|_{\mathbb{H}}^2.
	\end{align}
	For all $n=0, \cdots, N$ and $k=0, \cdots, K$, denote the error $\varepsilon _{n}^{(k)}:=E\left\|u(t_{n})-u_{n}^{(k)}\right\|_{\mathbb{H}}^2$.
	Combining ($\ref{eq16}$) and ($\ref{eq23}$) enables us to derive
	\begin{align*}
		\varepsilon _{n}^{(k)}\leq(1+\Delta T)\varepsilon _{n-1}^{(k)}+C\Delta T\varepsilon _{n-1}^{(k-1)}.
	\end{align*}
	Let $\zeta^{k}=(\varepsilon_1^{k},\varepsilon_2^{k},\cdots,\varepsilon_{N}^{k})^{T}$ . It follows from Lemma \ref{lemma4} that 
	\begin{align*}
		\zeta^{k}\leq C\Delta T M(1+C\Delta T)\zeta^{k-1}\leq C^{k} \Delta T^{k} M^{k}(1+C\Delta T)\zeta^{0}.
	\end{align*}
	Taking infinity norm and using Lemma \ref{lemma3} imply 
	\begin{align*}
		\sup\limits_{1\leq n \leq N}	\varepsilon _{n}^{(k)} &\leq(1+C\Delta T)^{N-k-1}C_{k}\Delta T^{k} C_{N-1}^{k}\sup\limits_{1\leq n \leq N}	\varepsilon _{n}^{(0)}\\
		&\leq \frac{C_{k}}{k!} \Delta T^{k} \prod \limits_{j=1}^k (N-j)\sup\limits_{1\leq n \leq N}	\varepsilon _{n}^{(0)}.			 
	\end{align*}
	\begin{equation*}
		\sup\limits_{1\leq n \leq N}\left\|u(t_{n})-u_{n}^{(k)}\right\|_{L_{2}(\Omega,\,\mathbb{H})} \leq C_{T,N}\Delta T^{k/2} \sup\limits_{1\leq n \leq N}\left\|u(t_{n})-u_{n}^{(0)}\right\|_{L_{2}(\Omega,\, \mathbb{H})},
	\end{equation*}
	This completes the proof.
\end{proof}
\hfill $\Box$
\section{Error estimate of stochastic exponential integrator as the fine integrator }\label{sec6}
In this section, the error we considered  is the solution by the proposed algorithm and the reference solution generated by the fine propagator $\mathcal{F}$. To begin with, we define the reference solution as follows.
\begin{definition}
	For all $n = 0,\ldots,N$, the reference solution is defined by  the fine propagator on each subinterval $[t_{n-1},t_{n}]$ 
	\begin{align}\label{eq24}
		u_{n}^{ref}&=\mathcal{F}(t_{n-1},t_{n},u_{n-1}^{ref}),\\
		u_{0}^{ref}&=u_{0}.\nonumber
	\end{align}
	Precisely, 
	\begin{align}\label{eq25}
		&u_{n-1,j}^{ref}=\mathcal{F}(t_{n-1,j-1},t_{n-1,j},u_{n-1,j-1}^{ref}), \, j=1,\cdots,J,\\
		&u_{n-1,0}^{ref}=u_{n-1}^{ref}.\nonumber
	\end{align}
\end{definition}

\begin{theorem}\label{theorem2}
	Let  Assumptions \ref{assump1}, \ref{assump2}, \ref{assump3} and \ref{assump4} hold,  we apply the stochastic exponential integrator for coarse propagator $\mathcal{G}$ and the stochastic exponential integrator for fine propagator $\mathcal{F}$. Then we have the following convergence estimate for the fixed iteration number $k$ 
\end{theorem}
\begin{align}\label{eq26}
	\sup\limits_{1\leq n \leq N}\left\|u_{n}^{(k)}-u_{n}^{ref}\right\|_{L_{2}(\Omega,\,\mathbb{H})}\leq  C_{T,N}\Delta T^{k/2} \left\|u_{n}^{(0)}-u_{n}^{ref}\right\|_{L_{2}(\Omega,\,\mathbb{H})},
\end{align}
with a positive constant $C$ independent on $\Delta T$, where the parareal solution $u_{n}^{(k)}$ is defined in (\ref{eq5}) and the reference solution $u_{n}^{ref}$ is defined in (\ref{eq24}).
\begin{proof} 
	Observe that the reference solution (\ref{eq24}) can be rewritten 
	\begin{align}\label{eq27}
		u_{n}^{ref}=\mathcal{F}(t_{n-1},t_{n},u_{n-1}^{ref})+\mathcal{G}(t_{n-1},t_{n},u_{n-1}^{ref})-\mathcal{G}(t_{n-1},t_{n},u_{n-1}^{ref}).
	\end{align}
	Combining the parareal algorithm form (\ref{eq5}) and the reference solution (\ref{eq27}) and using the notation of the residual operator (\ref{eq12}), the error can be written as
	\begin{align*}
		E\left\|u_{n}^{(k)}-u_{n}^{ref}\right\|_{\mathbb{H}}^2
		&=E\left\|\mathcal{G}(u_{n-1}^{(k)})-\mathcal{G}(u_{n-1}^{ref})+\mathcal{R}(u_{n-1}^{(k-1)})-\mathcal{R}(u_{n-1}^{ref})\right\|_{\mathbb{H}}^2\\
		&\leq C\{ E\left\|\mathcal{G}(u_{n-1}^{(k)})-\mathcal{G} (u_{n-1}^{ref})\right\|_{\mathbb{H}}^2+E\left\| \mathcal{R}(u_{n-1}^{(k-1)})
		- \mathcal{R}(u_{n-1}^{ref})\right\|_{\mathbb{H}}^2\}\\
		&:=I_{1}+I_{2}.
	\end{align*}
	Now we estimate $I_{1}$. Applying the stochastic exponential integrator (\ref{eq8}) for the coarse propagator $\mathcal{G}$, we obtain
	\begin{align}\label{eq28}
		\mathcal{G} (t_{n-1},t_{n},u_{n-1}^{ref})=\widehat S(\Delta T)u_{n-1}^{ref}+\widehat S(\Delta T)F(u_{n-1}^{ref})\Delta T+\widehat S(\Delta T)B(u_{n-1}^{ref})\Delta W_{n}.
	\end{align}
	Subtracting the above formula (\ref{eq28}) from (\ref{eq15}), we have
	\begin{align*}
		I_{1}\!=\!E\left\|\widehat S(\Delta T)(u_{n-1}^{(k)}-u_{n-1}^{ref})\!\!+\!\!\widehat S(\Delta T)[F(u_{n-1}^{(k)})\!-\!F(u_{n-1}^{ref})]\Delta T\!\!+\!\!\widehat S(\Delta T)[B(u_{n-1}^{(k)})\!-\!B(u_{n-1}^{ref})]\Delta W_{n}\right\|_{\mathbb{H}}^2\!.
	\end{align*}
	Armed with contraction property of semigroup and Lipschitz continuity property of $F$ and $B$ yield
	\begin{align}\label{eq29}
		I_{1}&\leq C\{E\left\|\widehat S(\Delta T)\right\|_{\mathcal{L(\mathbb{H})}}^2\left\|u_{n-1}^{(k)}-u_{n-1}^{ref}\right\|_{\mathbb{H}}^2+
		C\Delta T^2 E\left\|\widehat S(\Delta T)\right\|_{\mathcal{L(\mathbb{H})}}^2 \left\|u_{n-1}^{(k)}-u_{n-1}^{ref}\right\|_{\mathbb{H}}^2\nonumber\\
		&+C\Delta T \left\|Q^{\frac{1}{2}}\right\|_{HS}^2 E\left\|\widehat S(\Delta T)\right\|_{\mathcal{L(\mathbb{H})}}^2 \left\|u_{n-1}^{(k)}-u_{n-1}^{ref}\right\|_{\mathbb{H}}^2\}\nonumber\\
		&\leq C\{ E\left\|u_{n-1}^{(k)}-u_{n-1}^{ref}\right\|_{\mathbb{H}}^2+
		\Delta T^2 E\left\|u_{n-1}^{(k)}-u_{n-1}^{ref}\right\|_{\mathbb{H}}^2+\Delta T E\left\|u_{n-1}^{(k)}-u_{n-1}^{ref}\right\|_{\mathbb{H}}^2\}\nonumber\\
		&\leq C(1+\Delta T^2+\Delta T )E\left\|u_{n-1}^{(k)}-u_{n-1}^{ref}\right\|_{\mathbb{H}}^2\nonumber\\
		&\leq C(1+\Delta T )E\left\|u_{n-1}^{(k)}-u_{n-1}^{ref}\right\|_{\mathbb{H}}^2.
	\end{align}
	As for $I_{2}$, regarding the estimation of the residual operator, we need to resort to the boundedness of its  derivatives. Due to formula ($\ref{eq12}$), the derivatives in the direction \(h\) can be expressed as 
	\begin{align}\label{eq30}
		D\mathcal{R}(t_{n-1},t_{n},u).h:=D\mathcal{F}(t_{n-1},t_{n},u).h-D\mathcal{G}(t_{n-1},t_{n},u).h.
	\end{align}
	One the one hand, since the stochastic exponential scheme is chosen as the fine propagator (\ref{eq10}) with time step-size $\Delta t$, we obtain
	\begin{align*}
		\left\{
		\begin{array}{ll}
			u _{n,j+1}=e^{\Delta t(M-\sigma Id)}u _{n,j} +\Delta t e^{\Delta t(M-\sigma Id)}F(u_{n,j})+e^{\Delta t(M-\sigma Id)}B(u_{n,j})\Delta W_{n,j},\\
			u _{n,0}=u .
		\end{array}
		\right.
	\end{align*}
	Denote  $D(u_{n,j}).h:=\eta^{h}_{n,j} $ for $j \in 0,\cdots,J$. Then  taking the direction derivatives for above equation yields
	\begin{align*}
		\left\{
		\begin{array}{ll}
			\eta _{n,j+1}^{h}=e^{\Delta t(M-\sigma Id)}\eta _{n,j}^{h} +\Delta t e^{\Delta t(M-\sigma Id)}DF(u _{n,j}).\eta _{n,j}^{h}+e^{\Delta t(M-\sigma Id)}DB(u _{n,j}).\eta _{n,j}^{h}\Delta W_{n,j},\\
			\eta _{n,0}^{h}=h .
		\end{array}
		\right.
	\end{align*}
	We have the following recursion formula
	\begin{align*}
		\eta _{n,J}^{h}&\!=\!e^{J\Delta t(M-\sigma Id)}\eta _{n,0}^{h}\!\!+\!\!\Delta t\!\sum\limits_{j=0}\limits^{J-1}\!e^{(J-j)\Delta t(M-\sigma Id)}DF(u _{n,j}).\eta _{n,j}^{h}\!\!+\!\!\sum\limits_{j=0}\limits^{J-1}\!e^{(J-j)\Delta t(M-\sigma Id)}DB(u _{n,j}).\eta _{n,j}^{h}\Delta W_{n,j}\\
		&\!=\!e^{\Delta T(M-\sigma Id)}h\!\!+\!\Delta t\!\sum\limits_{j=0}\limits^{J-1}\!\!e^{(J-j)\Delta t(M-\sigma Id)}\!DF(u _{n,j}).\eta _{n,j}^{h}\!\!+\!\!\sum\limits_{j=0}\limits^{J-1}\!\!e^{(J-j)\Delta t(M-\sigma Id)}\!DB(u _{n,j}).\eta _{n,j}^{h}\Delta W_{n,j}.
	\end{align*}
	Utilizing the bounded derivatives condition of $F$ and $B$, we get
	\begin{align*}
		E\left\|\eta _{n,J}^{h}\right\|_{\mathbb{H}}^2&\leq C\{ E\left\|h\right\|_{\mathbb{H}}^2+J\Delta t^2E\sum\limits_{j=0}\limits^{J-1}\left\|\eta _{n,j}^{h}\right\|_{\mathbb{H}}^2+J \Delta t \left\|Q^{\frac{1}{2}}\right\|_{HS}^2\left\|Q^{\frac{1}{2}}\right\|_{HS}^2E\sum\limits_{j=0}\limits^{J-1}\left\|\eta _{n,j}^{h}\right\|_{\mathbb{H}}^2\}\\
		&\leq C \left\|h\right\|_{\mathbb{H}}^2+C(\Delta t\Delta T+\left\|Q^{\frac{1}{2}}\right\|_{HS}^2\left\|Q^{\frac{1}{2}}\right\|_{HS}^2\Delta T)E\sum\limits_{j=0}\limits^{J-1}\left\|\eta _{n,j}^{h}\right\|_{\mathbb{H}}^2\\
		&\leq C \left\|h\right\|_{\mathbb{H}}^2+C\Delta TE\sum\limits_{j=0}\limits^{J-1}\left\|\eta _{n,j}^{h}\right\|_{\mathbb{H}}^2.
	\end{align*}
	Applying the discrete Gronwall lemma yields the following inequality
	\begin{align}\label{eq31}
		\sup\limits_{1\leq n \leq N}E\left\|\eta _{n,j}^{h}\right\|_{\mathbb{H}}^2\leq C\left\|h\right\|_{\mathbb{H}}^2.
	\end{align}
	Moreover, the derivative of $\mathcal{F}(t_{n-1},t_{n},u)$ can be writen by $D\mathcal{F}(t_{n-1},t_{n},u).h=D(u_{n,J}).h=\eta_{n,J}^{h}$, where $J\Delta t=\Delta T$, that is, one gets
	\begin{align}\label{eq32}
		D\mathcal{F}(t_{n-1},t_{n},u).h=e^{\Delta T(M-\sigma Id)}h+\Delta t\sum\limits_{j=0}\limits^{J-1}e^{(J-j)\Delta t(M-\sigma Id)}DF(u_{n,j}). \eta _{n,j}^{h}\nonumber\\
		+\sum\limits_{j=0}\limits^{J-1}e^{(J-j)\Delta t(M-\sigma Id)}DB(u _{n,j}).\eta _{n,j}^{h}\Delta W_{n,j}.
	\end{align}
	On the other hand, since the stochastic exponential scheme is chosen as the coarse propagator $\mathcal{G}$, taking the direction derivative for $u$ of formula (\ref{eq8}) leads to
	\begin{align}\label{eq33}
		D\mathcal{G}(t_{n-1},t_{n},u).h=e^{\Delta T(M-\sigma Id)}h+\Delta Te^{\Delta T(M-\sigma Id)}DF(u).h+e^{\Delta T(M-\sigma Id)}DB(u).h\Delta W_{n}.
	\end{align}
	Substituting formula (\ref{eq32}) and (\ref{eq33}) into formula (\ref{eq30}), we obtain
	\begin{align*}
		E\left\|D\mathcal{R}(t_{n-1},t_{n},u).h\right\|_{\mathbb{H}}^2&= E\left\|D\mathcal{F}(t_{n-1},t_{n},u).h-D\mathcal{G}(t_{n-1},t_{n},u).h\right\|_{\mathbb{H}}^2\\
		&\!\!\!\!\!\!\!\!\!\!\!\!\!\leq C\{E\!\left\|\Delta t\!\sum\limits_{j=0}\limits^{J-1}e^{(J-j)\Delta t(M-\sigma Id)}DF(u _{n,j}). \eta _{n,j}^{h}\right\|_{\mathbb{H}}^2\!\!\!\!+\!\!E\!\left\|\Delta Te^{\Delta T(M-\sigma Id)}DF(u).h\right\|_{\mathbb{H}}^2 \\
		&\!\!\!\!\!\!\!\!\!\!\!\!\!\!\!+\!\!E\left\|\sum\limits_{j=0}\limits^{J-1}e^{(J-j)\Delta t(M-\sigma Id)}DB(u _{n,j}).\eta _{n,j}^{h}\Delta W_{n,j}\right\|_{\mathbb{H}}^2\!\!\!\!\!+\!\!E\left\|e^{\Delta T(M-\sigma Id)}DB(u).h\Delta W_{n}\right\|_{\mathbb{H}}^2\}.
	\end{align*}
	Utilizing the bounded derivatives condition of $F$ and $B$, we get
	\begin{align*}
		\!\!\!\!E\left\|D\mathcal{R}(t_{n-1},t_{n},u).h\right\|_{\mathbb{H}}^2
		&\!\leq\!C\{J\Delta t^2E\sum\limits_{j=0}\limits^{J-1}\left\|e^{(J-j)\Delta t(M-\sigma Id)}\right\|_{\mathcal{L(\mathbb{H})}}^2\!\!\left\|\eta _{n,j}^{h}\right\|_{\mathbb{H}}^2\!\!+\!\Delta T^2E\left\|e^{\Delta T(M-\sigma Id)}\right\|_{\mathcal{L(\mathbb{H})}}^2\!\!\left\|h\right\|_{\mathbb{H}}^2\\
		&+J\Delta t \left\|Q^{\frac{1}{2}}\right\|_{HS}^2\left\|Q^{\frac{1}{2}}\right\|_{HS}^2E\sum\limits_{j=0}\limits^{J-1}\left\|e^{(J-j)\Delta t(M-\sigma Id)}\right\|_{\mathcal{L(\mathbb{H})}}^2\left\|\eta _{n,j}^{h}\right\|_{\mathbb{H}}^2\\
		&+\Delta T\left\|Q^{\frac{1}{2}}\right\|_{HS}^2\left\|Q^{\frac{1}{2}}\right\|_{HS}^2E\left\|e^{\Delta T(M-\sigma Id)}\right\|_{\mathcal{L(\mathbb{H})}}^2\left\|h\right\|_{\mathbb{H}}^2\}.
	\end{align*}
	Using the contraction property of semigroup, we have
	\begin{align*}
		E\left\|D\mathcal{R}(t_{n-1},t_{n},u).h\right\|_{\mathbb{H}}^2
		&\leq C\{J\Delta t^2\sum\limits_{j=0}\limits^{J-1}\sup\limits_{1\leq n \leq N}E\left\|\eta _{n,j}^{h}\right\|_{\mathbb{H}}^2+\Delta T^2E\left\|h\right\|_{\mathbb{H}}^2\\
		&\!\!\!\!\!\!\!\!\!+\Delta T\left\|Q^{\frac{1}{2}}\right\|_{HS}^2\left\|Q^{\frac{1}{2}}\right\|_{HS}^2\sum\limits_{j=0}\limits^{J-1}\sup\limits_{1\leq n \leq N}E\left\|\eta _{n,j}^{h}\right\|_{\mathbb{H}}^2+\Delta T\left\|Q^{\frac{1}{2}}\right\|_{HS}^2\left\|Q^{\frac{1}{2}}\right\|_{HS}^2E\left\|h\right\|_{\mathbb{H}}^2\}.
	\end{align*}
	Substituting the Gronwall inequality ($\ref{eq31}$) into the above inequality leads to
	\begin{align*}
		\sup\limits_{1\leq n \leq N}E\left\|D\mathcal{R}(t_{n-1},t_{n},u).h\right\|_{\mathbb{H}}^2
		&\leq C\{\Delta t^2 J^2 \left\|h\right\|_{\mathbb{H}}^2+\Delta T^2\left\|h\right\|_{\mathbb{H}}^2\\
		&+J\Delta T \left\|Q^{\frac{1}{2}}\right\|_{HS}^2\left\|Q^{\frac{1}{2}}\right\|_{HS}^2 \left\|h\right\|_{\mathbb{H}}^2+\Delta T\left\|Q^{\frac{1}{2}}\right\|_{HS}^2\left\|Q^{\frac{1}{2}}\right\|_{HS}^2\left\|h\right\|_{\mathbb{H}}^2\}\\
		&\leq C(\Delta T^2+\Delta T^2+\Delta T+\Delta T)\left\|h\right\|_{\mathbb{H}}^2\\
		&\leq C\Delta T\left\|h\right\|_{\mathbb{H}}^2.
	\end{align*}
	In conclusion, it holds that
	\begin{align}\label{eq34}
		\sup\limits_{1\leq n \leq N}E\left\|\mathcal{R}(t_{n-1},t_{n},u_{2})-\mathcal{R}(t_{n-1},t_{n},u_{1})\right\|_{\mathbb{H}}^2\leq C \Delta TE\left\|u_{2}-u_{1}\right\|_{\mathbb{H}}^2,\,\forall u_{1}, u_{2}\in \mathbb{H}.
	\end{align}
	Substituting $u_{n-1}^{(k-1)}$ and $u_{n-1}^{ref}$ into above formula derives lipschitz continuity property of the residual operator
	\begin{align}\label{eq35}
		I_{2}&=E\left\|\mathcal{R}(t_{n-1},t_{n},u_{n-1}^{(k-1)})-\mathcal{R}(t_{n-1},t_{n},u_{n-1}^{ref})\right\|_{\mathbb{H}}^2\leq C\Delta T E\left\|u_{n-1}^{(k-1)}-u_{n-1}^{ref}\right\|_{\mathbb{H}}^2.
	\end{align}
	For all $n = 0,\cdots,N$ and $k = 0,\cdots,K$, let the error be defined by 
	$\varepsilon _{n}^{(k)}:=	E\left\|u_{n}^{k}-u_{n}^{ref}\right\|_{\mathbb{H}}^2$. Combining (\ref{eq29}) and (\ref{eq35}), we have
	\begin{align*}
		\varepsilon _{n}^{(k)}
		\leq(1+C\Delta T)\varepsilon _{n-1}^{(k)}+C\Delta T \varepsilon _{n-1}^{(k-1)}.
	\end{align*}
	According to Lemma \ref{lemma3} and Lemma \ref{lemma4}, it yields to
	\begin{align*}
		\sup\limits_{1\leq n \leq N}	\varepsilon _{n}^{(k)} &\leq(1+C\Delta T)^{N-k-1}C_{k}\Delta T^{k} C_{N-1}^{k}\sup\limits_{1\leq n \leq N}	\varepsilon _{n}^{(0)}\\
		&\leq \frac{C_{k}}{k!} \Delta T^{k} \prod \limits_{j=1}^k (N-j)\sup\limits_{1\leq n \leq N}	\varepsilon _{n}^{(0)},		 
	\end{align*}
	which leads to the final result
	\begin{align*}
		\sup\limits_{1\leq n \leq N}\left\|u_{n}^{(k)}-u_{n}^{ref}\right\|_{L_{2}(\Omega,\,\mathbb{H})}\leq C_{T,N}\Delta T^{k/2} \left\|u_{n}^{(0)}-u_{n}^{ref}\right\|_{L_{2}(\Omega,\,\mathbb{H})}.
	\end{align*}
	The proof is thus completed.
\end{proof}\hfill $\Box$
\section{Numerical experiments}\label{sec7}
In this section, we present some numerical experiments to illustrate the theoretical results about the parareal algorithm, mainly focusing on the convergence rates and computational efficiency. Without loss of generality, we consider two-dimensioanl stochastic Maxwell equations (\ref{eq2}) with TM polarization on the domain $[0,1]\times[0,1]$, i.e., the electric field and the magnetic field are $\bm{E}=(0,0,E_{z})$ and $\bm{H}=(H_{x},H_{y},0)$.

The preset initial conditions are as follows:
\begin{align*}
	E_{z}(x,y,0)&=0.1exp(-50((x-0.5)^2+(y-0.5)^2)),\\
	H_{x}(x,y,0)&=rand_{y},\\
	H_{y}(x,y,0)&=rand_{x},
\end{align*}
where  $rand_{x}$ and $rand_{y}$ represent random initial values in one direction, while the other direction is kept constant. 

\begin{figure}[htbp]
	\centerline{\includegraphics[width=5in,height=2.5in]{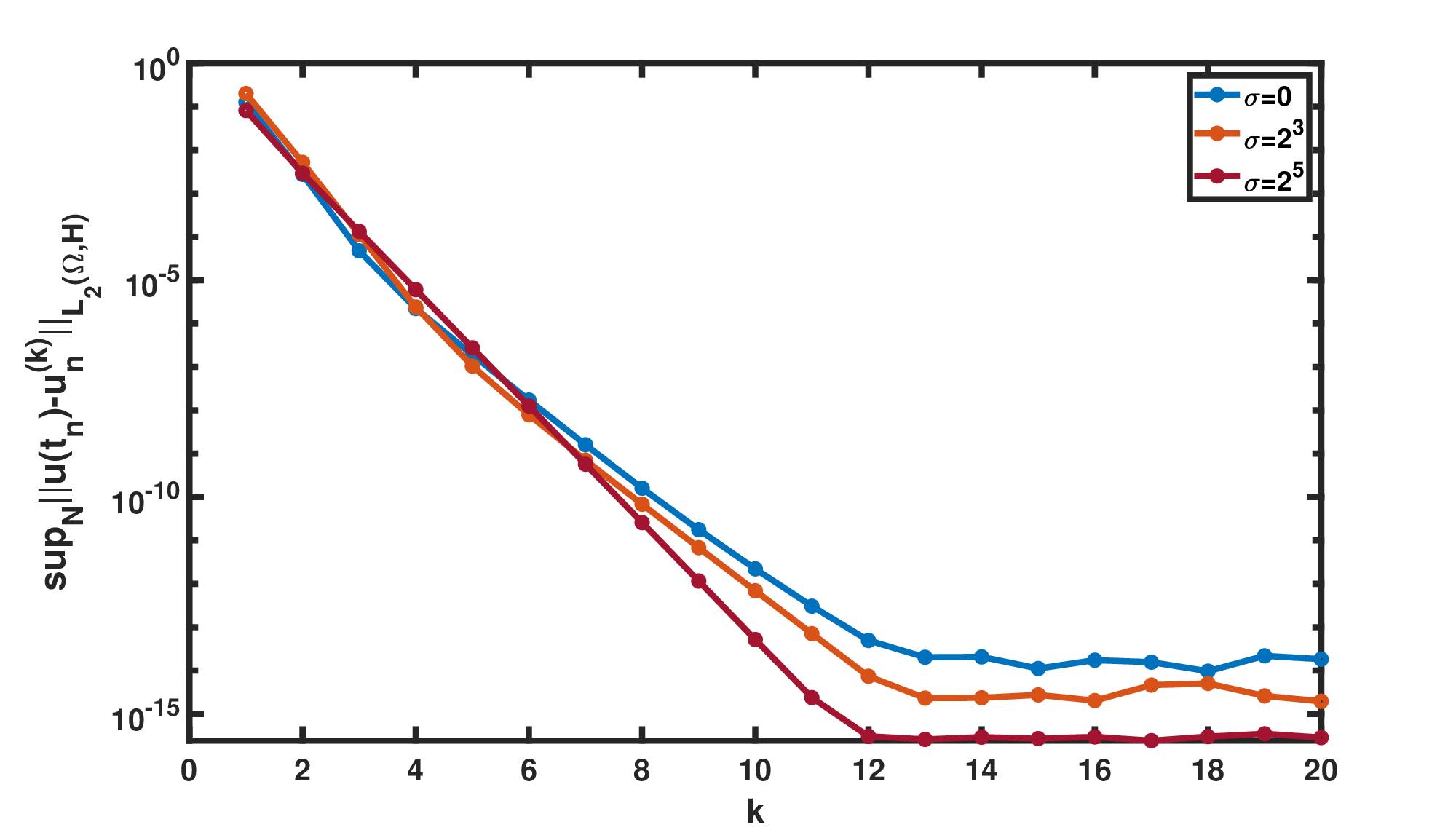}}
	\vspace*{8pt}
	\caption{Convergence with interation number $k$ for different values of $\sigma=0,2^1,2^3,2^5$}
	\label{fig1}
\end{figure}

\begin{figure}[htbp]
	\centerline{\includegraphics[width=5in,height=2.5in]{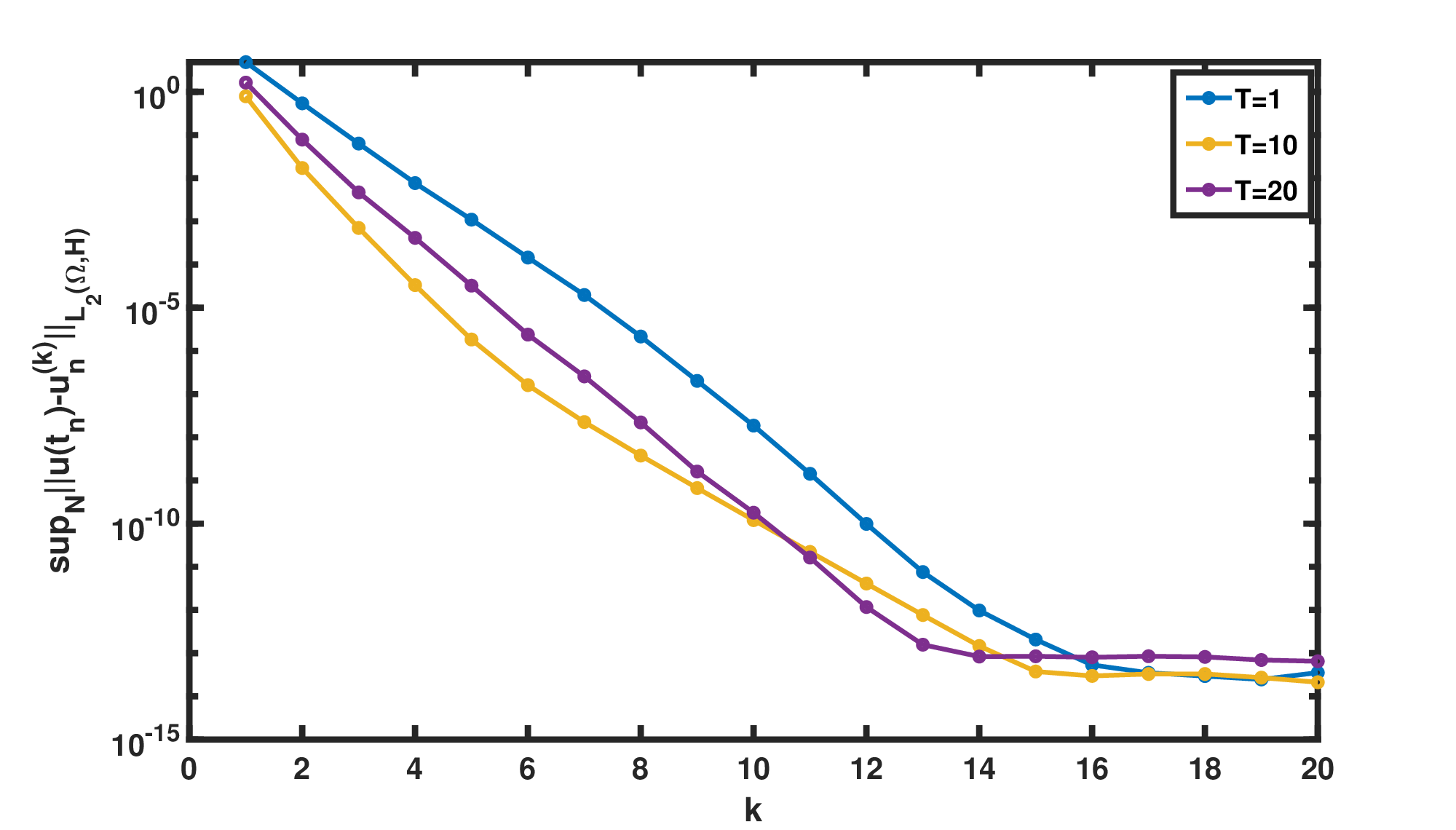}}
	\vspace*{8pt}
	\caption{Convergence for different time $T=1,10,20$}
	\label{fig2}
\end{figure}

Firstly, we focus on analyzing the convergence behavior of the parareal algorithm for different values of the iteration number $k$ and the damping coefficient $\sigma$. The algorithm is applied to solve the numerical solution with the fine step-size $\Delta t = 2^{-8}$, the coarse step-size $\Delta T = 2^{-6}$ and the spatial step-size $\Delta x = \Delta y = 2^{-4}$.  Figures \ref{fig1} and \ref{fig2} present the evolution of the mean-square error  with respect to the iteration number $k$. As illustrated in Figure \ref{fig1}, we observe that the proposed algorithm converges and the damping term accelerates the convergence of the numerical solutions. To assess the stability of the proposed algorithm for long-term computations, we investigate scenarios with  $T = 1, 10, 20$. Figure \ref{fig2} demonstrates that the errors in the parareal algorithm remain consistently below \( 10^{-12} \) after \( k = 14 \) iterations, underscoring the stability of the proposed algorithm even during long-term computations.

\begin{figure}[htpb]
	\centering
	\begin{minipage}[b]{0.33\textwidth}
		\centering
		\includegraphics[width=3.5in,height=2.2in]{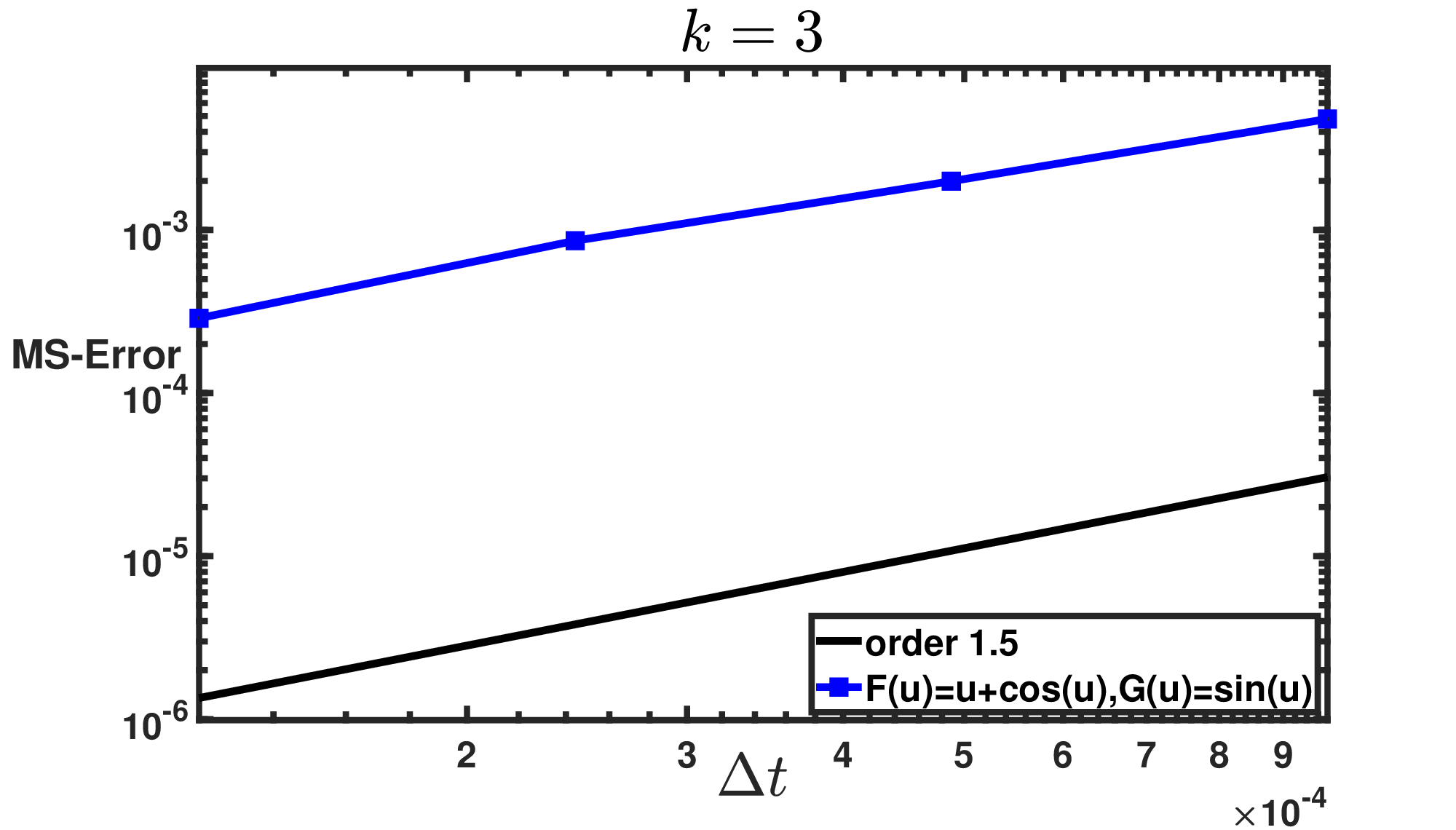}
	\end{minipage}
	\hfill 
	\begin{minipage}[b]{0.52\textwidth}
		\centering
		\includegraphics[width=3.8in,height=2.2in]{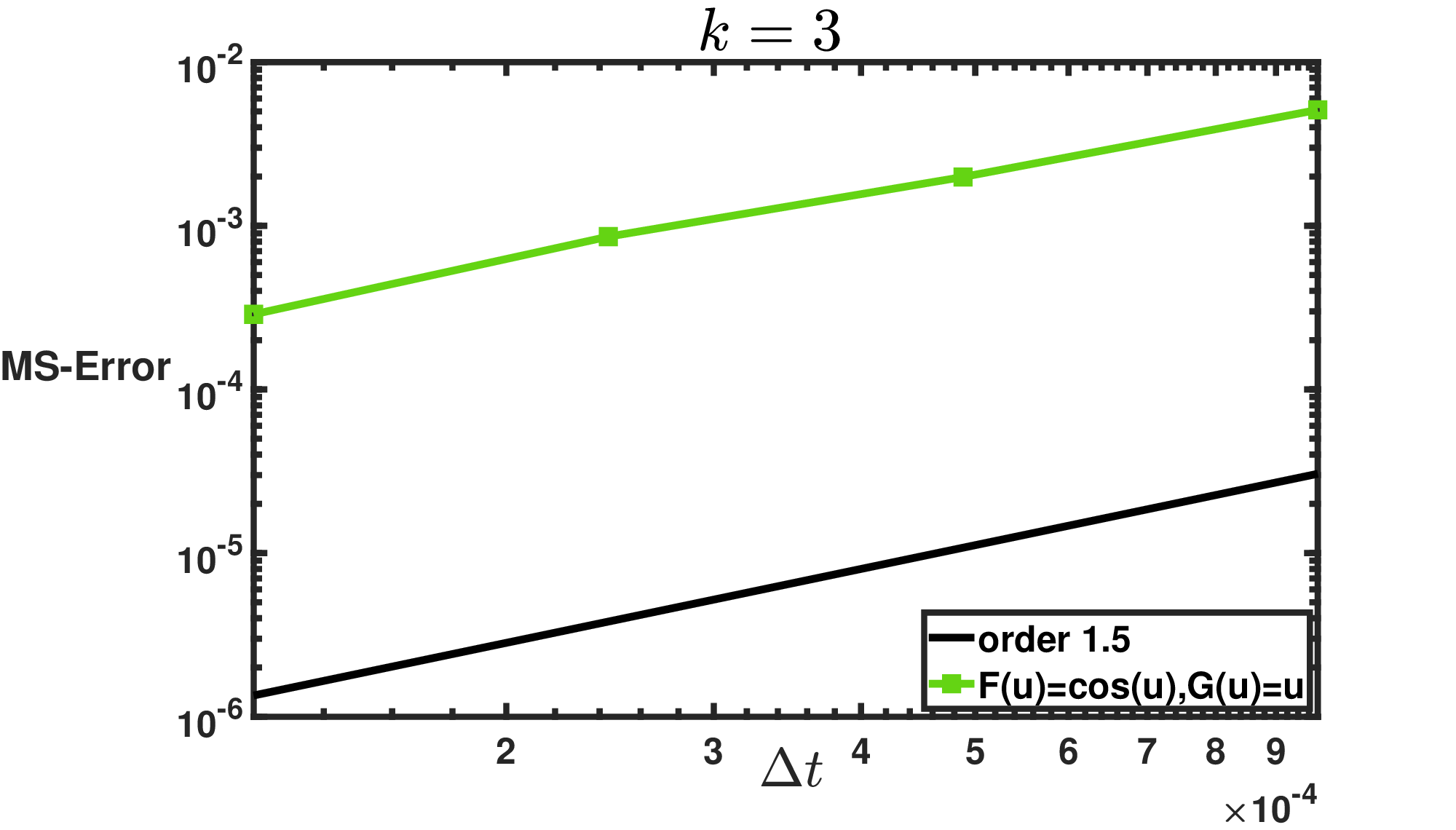}
	\end{minipage}
	\caption{Mean-square order for $k=3$ in the temporal direction in the cases of (left)$F(u)=u+cos(u),G(u)=sin(u)$ and (right)$F(u)=cos(u),G(u)=u$}
	\label{fig3}
\end{figure}
\begin{figure}[htpb]
	\centering
	\begin{minipage}[b]{0.33\textwidth}
		\centering
		\includegraphics[width=3.5in,height=2.2in]{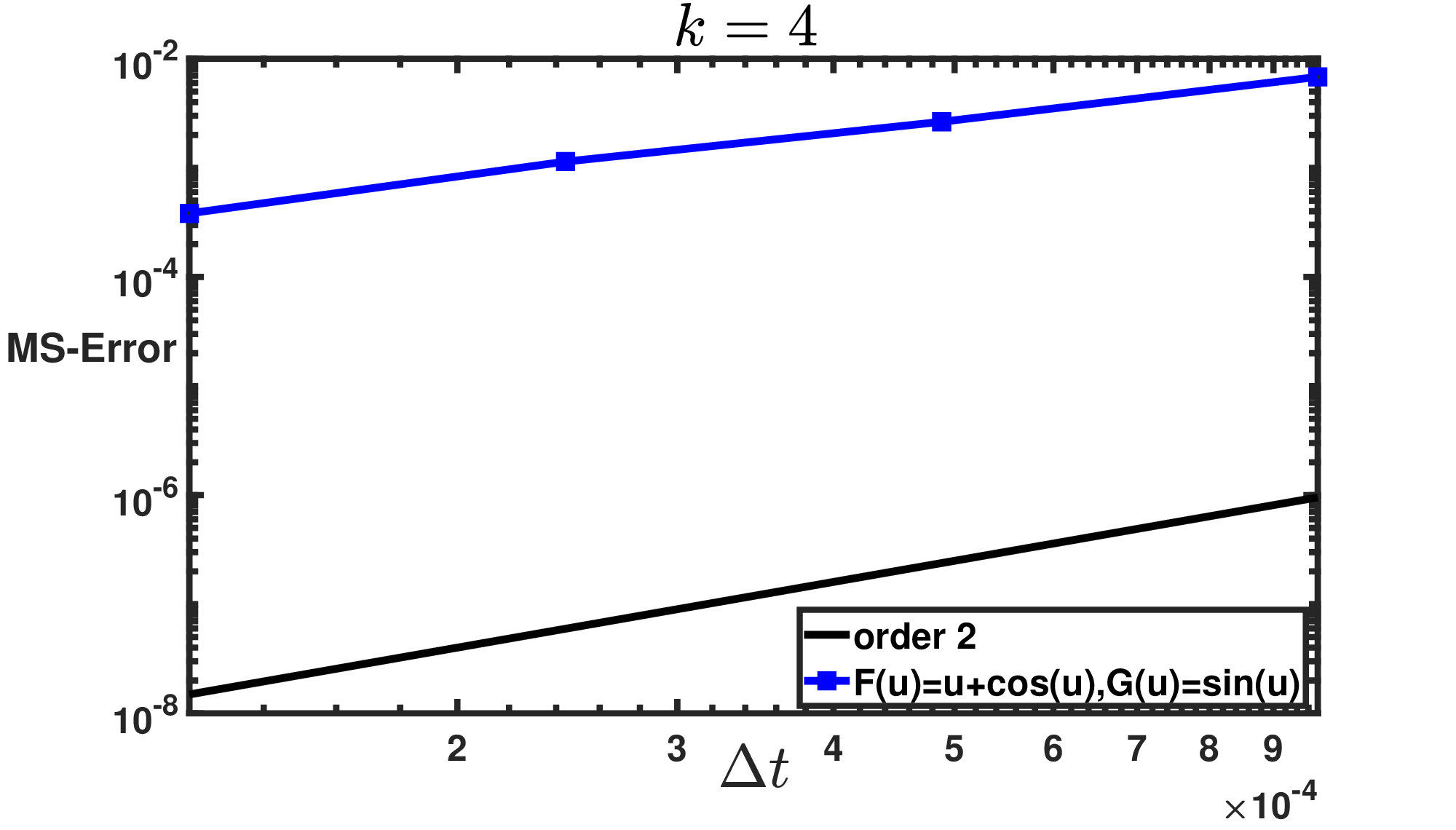}
	\end{minipage}
	\hfill 
	\begin{minipage}[b]{0.52\textwidth}
		\centering
		\includegraphics[width=3.8in,height=2.2in]{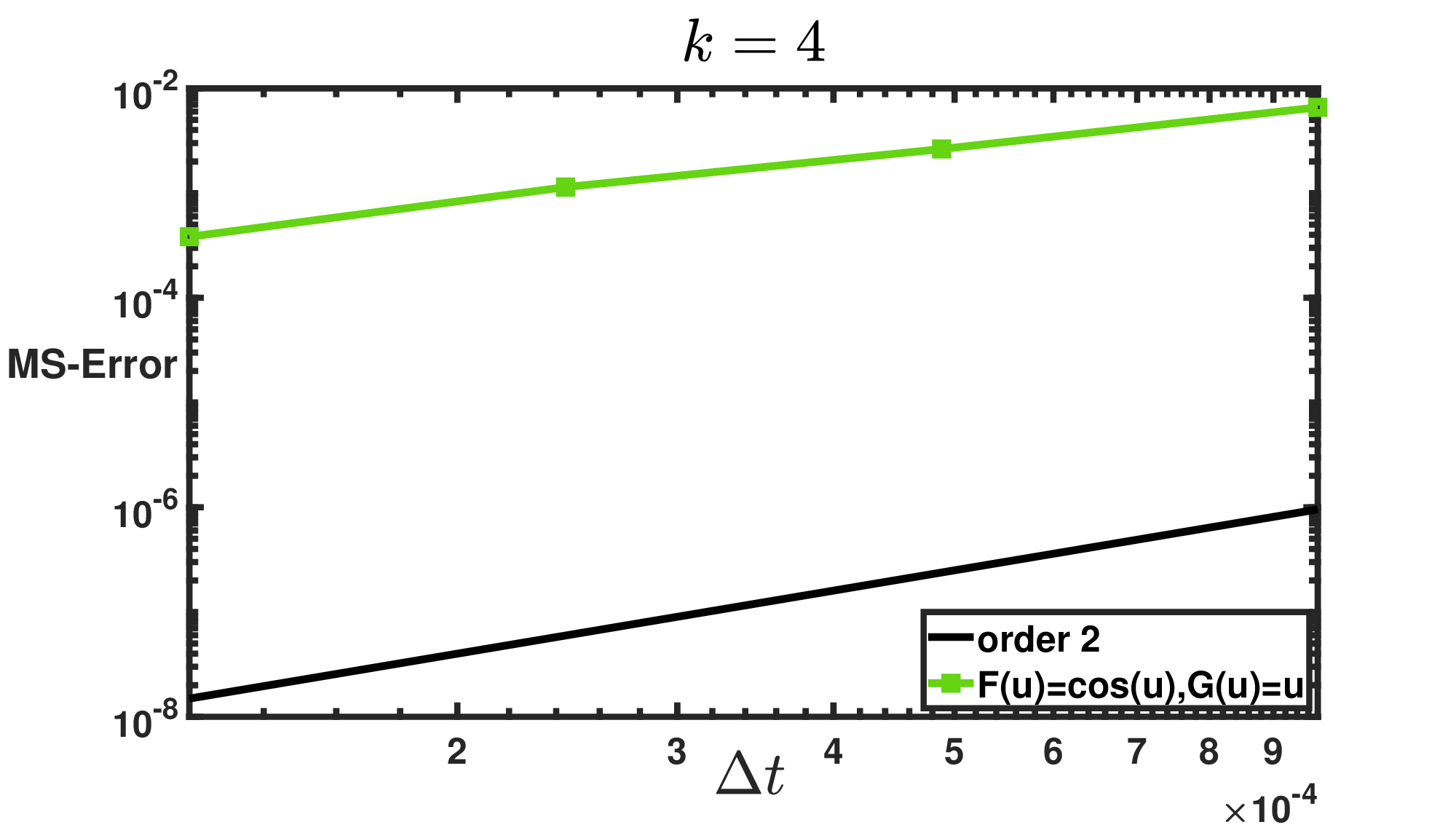}
	\end{minipage}
	\caption{Mean-square order for $k=4$ in the temporal direction in the cases of (left)$F(u)=u+cos(u),G(u)=sin(u)$ and (right)$F(u)=cos(u),G(u)=u$}
	\label{fig4}
\end{figure}
\begin{figure}[htpb]
	\centering
	\begin{minipage}[b]{0.33\textwidth}
		\centering
		\includegraphics[width=3.5in,height=2.2in]{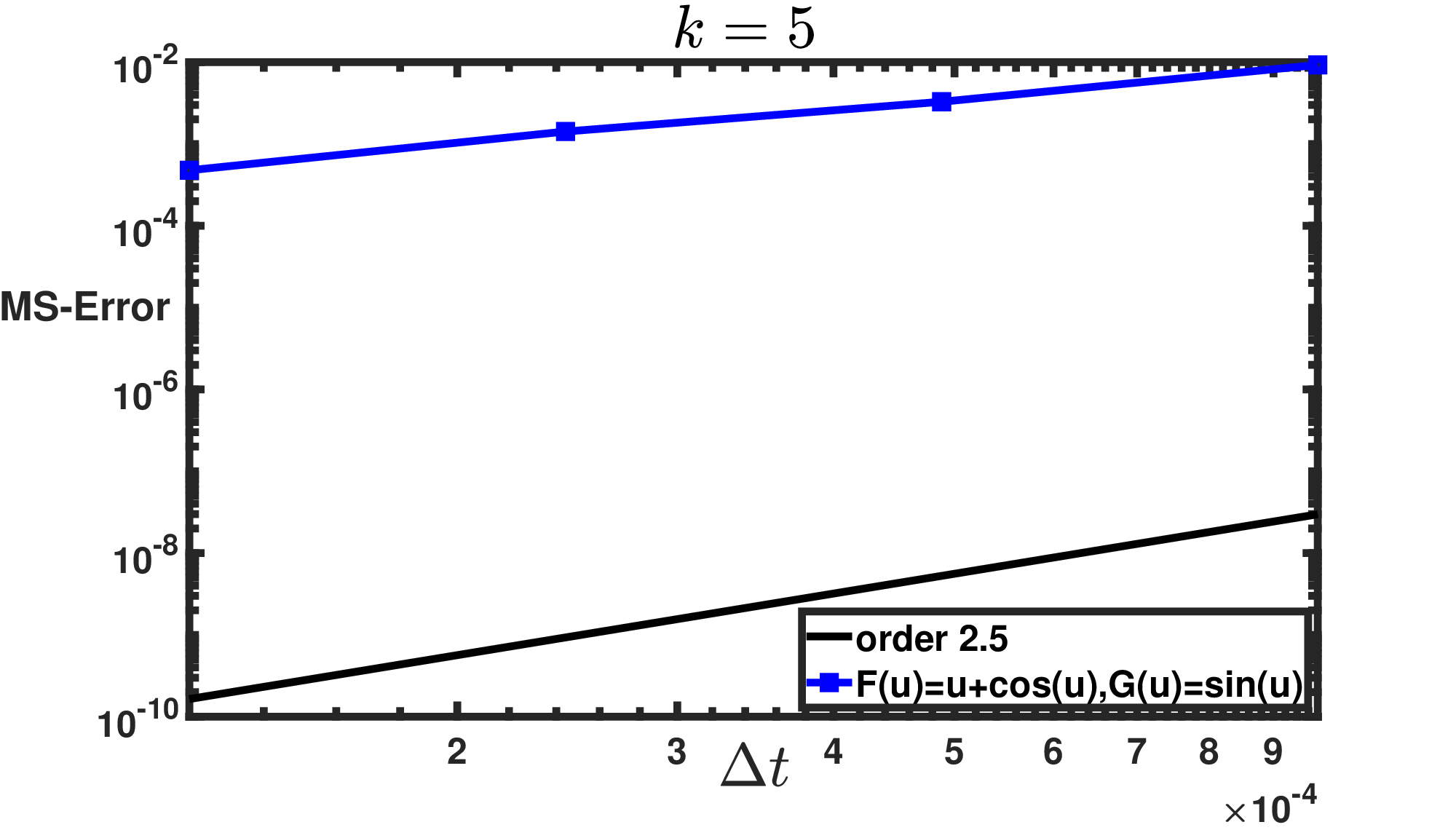}
	\end{minipage}
	\hfill 
	\begin{minipage}[b]{0.52\textwidth}
		\centering
		\includegraphics[width=3.8in,height=2.2in]{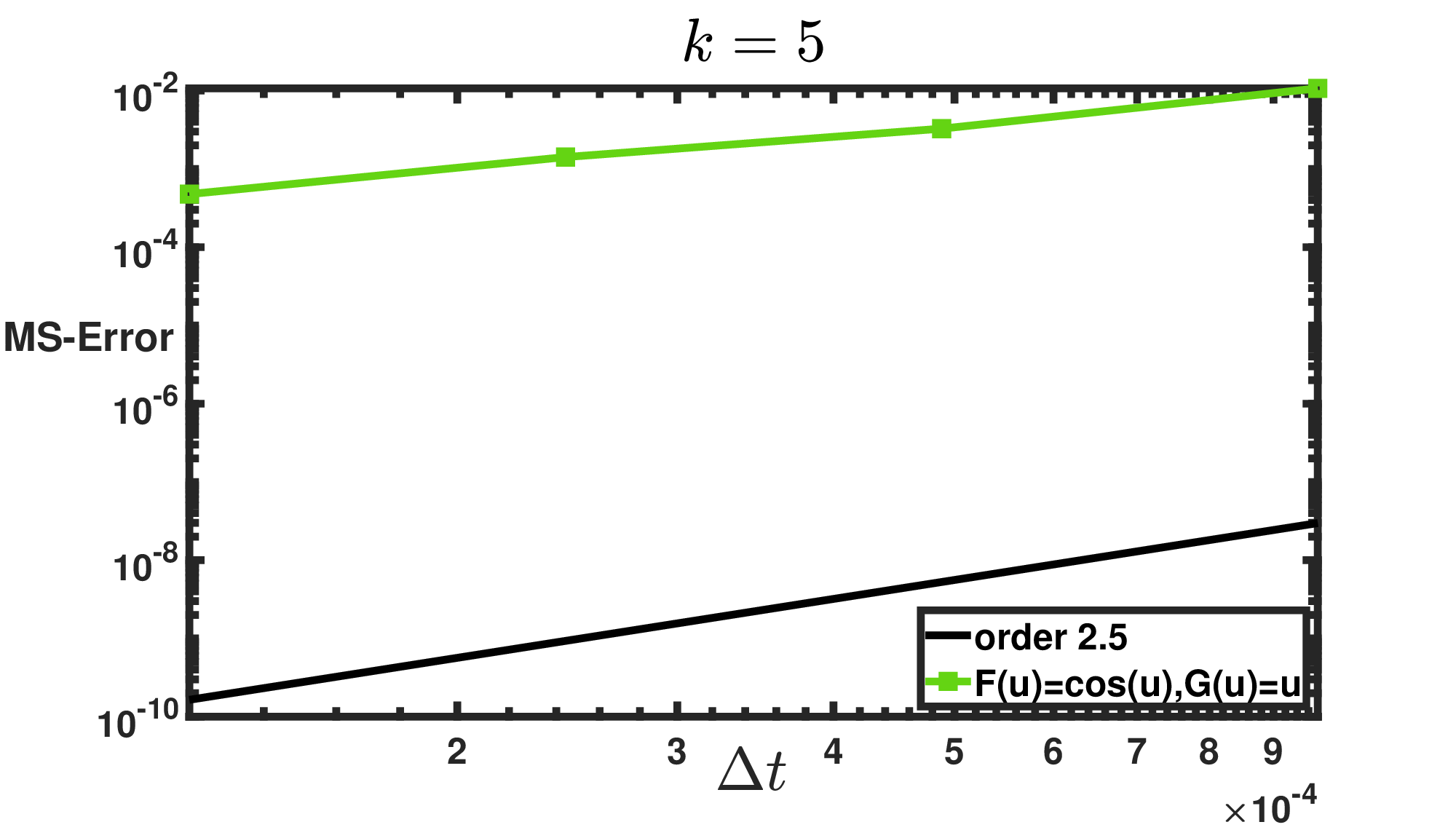}
	\end{minipage}
	\caption{Mean-square order for $k=5$ in the temporal direction in the cases of (left)$F(u)=u+cos(u),G(u)=sin(u)$ and (right)$F(u)=cos(u),G(u)=u$}
	\label{fig5}
\end{figure}

Further, we evaluate the convergence order by computing numerical solutions with the fine step-size  $\Delta t= 2^{-14}$ and  a series of coarse step-sizes $\Delta T=2^{-10},2^{-11},2^{-12},2^{-13}$. The convergence order, as presented in Figures \ref{fig3}, \ref{fig4}, \ref{fig5}, is consistent with the iteration number $k/2$.

The computational efficiency, evaluated by CPU time, varies significantly between the two methods. The parareal algorithm shows a marked advantage in CPU time for small time intervals $T$, as evident from Table \ref{table1}. However, as time $T$ increases, the computational cost of the exponential method becomes more pronounced due to the smaller time steps $\Delta T' = 10^{-2}$ or $10^{-4}$. Figures \ref{fig6} and \ref{fig7} illustrate the CPU time required by the parareal algorithm (\(k = 2, 3\)) and the exponential method for solving stochastic Maxwell equations at different simulation times \(T\). In Figure \ref{fig6}, the exponential method is used with a time step of \(\Delta T = 10^{-2}\) over a longer simulation time, while in Figure \ref{fig7}, the exponential method is applied with a smaller time step of \(\Delta T = 10^{-4}\) over a shorter simulation time.

\begin{figure}[htbp]
	\centerline{\includegraphics[width=5in,height=2.5in]{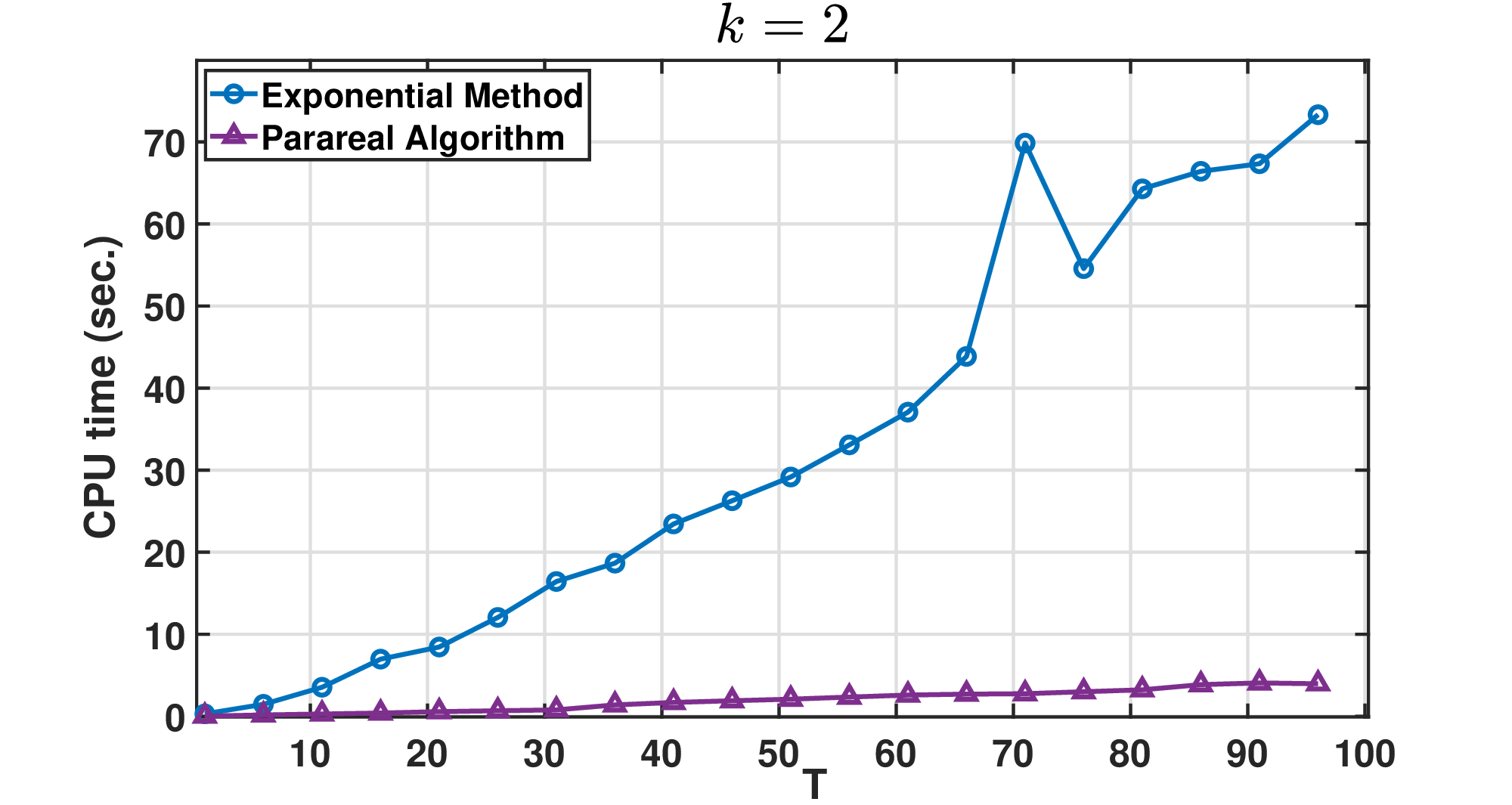}}
	\vspace*{8pt}
	\caption{The CPU times of different methods at different time $T$ for $k=2$}
	\label{fig6}
\end{figure}

\begin{figure}[htbp]
	\centerline{\includegraphics[width=5in,height=2.5in]{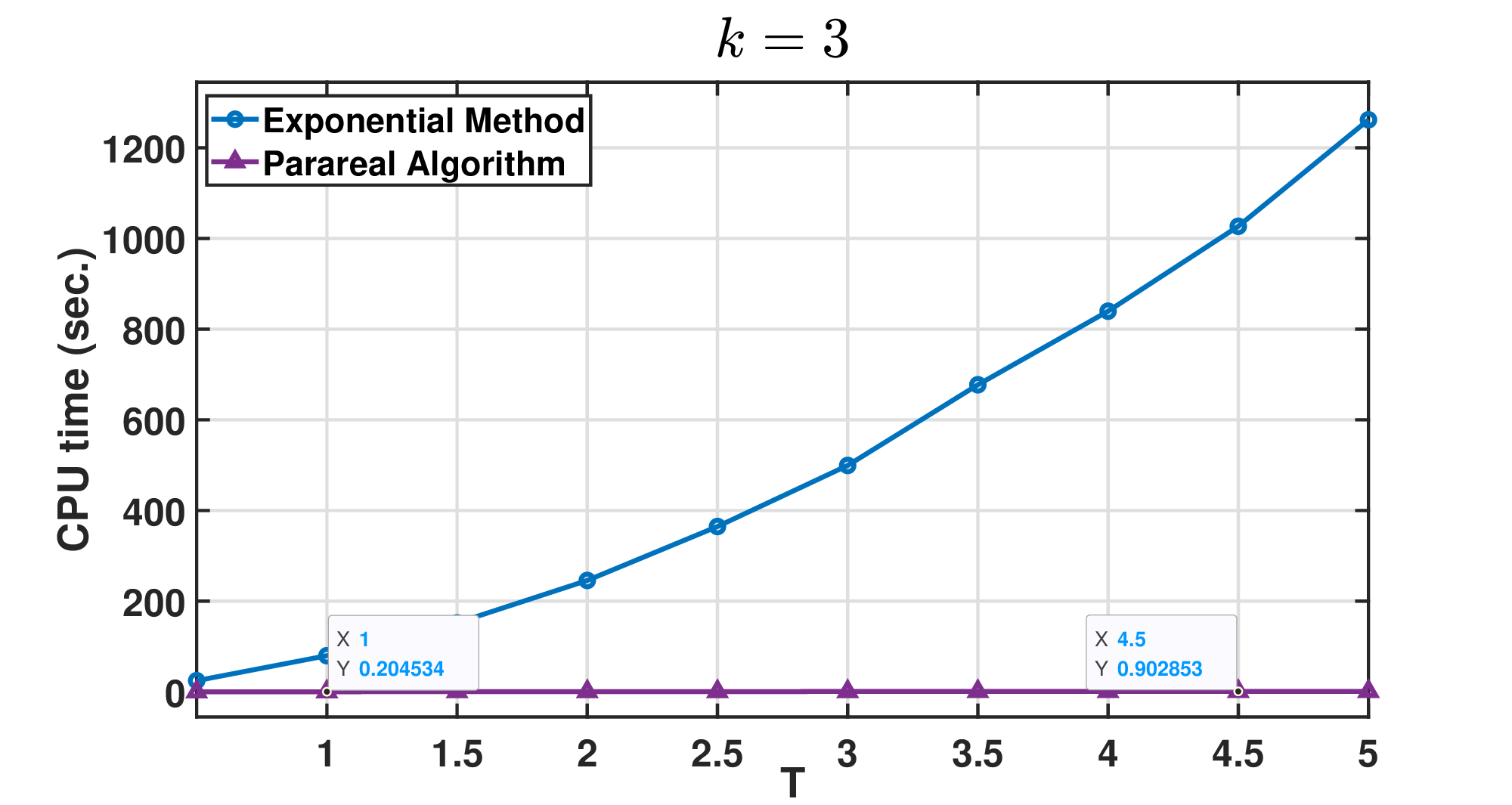 }}
	\vspace*{8pt}
	\caption{The CPU times of different methods at different time $T$ for $k=3$}
	\label{fig7}
\end{figure}

\begin{table}[h]
	\centering
	\begin{tabular}{|c|c|c|c|c|}
		\hline Method &  $\Delta T$  & $T$ &  $\|e\|_{2}$  & CPU time (sec.) \\
		\hline \multirow{4}{*}{Parareal $(k=2)$}
		& \multirow{4}{*}{$10^{-1}$} & 1  & 0.7216 E-2 & 0.0996 \\
		& & 10 & 0.5716 E-1 & 0.8133 \\
		& & 50 & 0.5515 E-1 & 5.5467 \\
		& & 100 & 1.1831 E-1 & 1.2866E+1 \\
		\hline \multirow{4}{*}{Exponential} 
		& \multirow{4}{*}{$10^{-2}$} &  1  & 5.109 E-1 & 0.2732 \\
		& &  10 & 5.230 E-1 & 2.4981 \\
		& & 50  & 5.0694 E-1 & 4.1426E+1 \\
		& & 100 & 5.1325 E-1 & 1.0311E+2 \\
		\hline 
	\end{tabular}
	\caption{Efficiency of parareal algorithm  and exponential method for the stochastic Maxwell equations  at different time $T=1,10,50,100$}
	\label{table1}
\end{table}
\begin{table}[htbp]
	\centering
	\begin{tabular}{|c|c|c|c|c|}
		\hline Method &  $\Delta T$  & $T$ &  $\|e\|_{2}$  & CPU time (sec.) \\
		\hline \multirow{4}{*}{Parareal $(k=3)$}
		& \multirow{4}{*}{$10^{-2}$} & 0.5 & 0.9203 E-2 & 0.3286  \\
		& & 1  & 0.2897 E-2 & 0.5028\\
		& & 5  & 1.3183 E-3 & 4.8440 \\
		& & 10 & 7.7695 E-4 & 1.0103E+1\\
		\hline \multirow{4}{*}{Exponential} 
		& \multirow{4}{*}{$10^{-4}$} & 0.5 & 0.5542 E-2 & 3.3512E+1 \\
		& & 1  & 0.5518 E-2 & 9.7711E+1 \\
		& & 5  & 0.5535 E-2 & 1.6306E+3 \\
		& & 10 & 0.5531 E-2 & 5.5174E+3 \\
		\hline 
	\end{tabular}
	\caption{Efficiency of parareal algorithm  and exponential method for the stochastic Maxwell equations  at different time $T=0.5,1,5,10$}
	\label{table2}
\end{table}

The  mean-square error $\|e\|_2$ between the solutions obtained by the parareal algorithm and the exponential method are detailed in Tables \ref{table1} and \ref{table2}. For both methods, $\|e\|_2$ decreases as the time step-size $\Delta T$ becomes smaller. However, the parareal algorithm demonstrates a significantly lower error at comparable time intervals $T$ due to its iterative correction mechanism, especially with increased iterations.
\section{Conclusion}\label{sec8}
In this paper, we propose the parareal algorithm for solving the stochastic Maxwell equations driven by multiplicative noise, where the stochastic exponential integrator is used as the coarse propagator, and both the exact integrator and the stochastic exponential integrator are used as the fine propagators. The algorithm significantly improves the convergence rate, achieving the mean-square convergence order of $k/2$. Compared to traditional methods that require smaller time steps to maintain accuracy, the parareal algorithm can achieve the same or higher precision with larger coarse time steps, thereby significantly reducing computational costs. Numerical experiments have verified the algorithm's efficiency and stability, particularly demonstrating superior performance in long-time simulations.
\section*{Acknowledgments}
The authors would like to express their appreciation to the referees for their useful comments and the editors. Liying Zhang is supported by 
the National Natural Science Foundation of China (No.11601514 and No.11971458), the Fundamental Research Funds for the Central Universities (No.2023ZKPYL02 and No.2023JCCXLX01) and the Yueqi Youth Scholar Research Funds for the China University of Mining and Technology-Beijing (No.2020YQLX03). Lihai Ji is supported by the National Natural Science Foundation of China (No.12171047).
\bibliography{refs}
\bibliographystyle{plain}
\end{document}